\documentclass[11pt,reqno]{amsart}
\usepackage[tmargin=1.5in,bmargin=1.4in,rmargin=1.4in,lmargin=1.4in]{geometry}
\usepackage[breaklinks=true]{hyperref}
\allowdisplaybreaks

\usepackage{amsmath,amssymb,amsthm}

\newcommand{\ba}{\mathbf{a}}
\newcommand{\balpha}{{\boldsymbol{\alpha}}}
\newcommand{\bb}{\mathbf{b}}
\newcommand{\bbeta}{{\boldsymbol{\beta}}}
\newcommand{\bc}{\mathbf{c}}
\newcommand{\bd}{\mathbf{d}}
\newcommand{\be}{\mathbf{e}}
\newcommand{\bj}{\mathbf{j}}
\newcommand{\bk}{\mathbf{k}}
\newcommand{\blambda}{{\boldsymbol{\lambda}}}
\newcommand{\bm}{\mathbf{m}}
\newcommand{\bmu}{{\boldsymbol{\mu}}}
\newcommand{\bn}{\mathbf{n}}
\newcommand{\bnu}{{\boldsymbol{\nu}}}
\newcommand{\bq}{\mathbf{q}}

\newcommand{\bx}{\mathbf{x}}
\newcommand{\bJ}{\mathbf{1}}
\newcommand{\bZ}{\mathbf{0}}
\newcommand{\C}{\mathbb{C}}
\newcommand{\cB}{\mathcal{B}}
\newcommand{\cK}{\mathcal{K}}
\newcommand{\cN}{\mathcal{N}}
\newcommand{\tcN}{\widetilde{\cN}}
\newcommand{\cZ}{\mathcal{Z}}
\newcommand{\diag}{\mathop{\mathrm{diag}}}
\newcommand{\E}{\mathbb{E}}
\newcommand{\eval}{\mathop{\mathrm{eval}}\nolimits}
\newcommand{\F}{\mathbb{F}}
\newcommand{\I}{\mathrm{i}}
\newcommand{\intd}{\,\rd}
\newcommand{\Q}{\mathbb{Q}}
\newcommand{\R}{\mathbb{R}}
\renewcommand{\Re}{\mathop{\mathrm{Re}}}
\newcommand{\ROHP}{\mathbb{H}^+}
\newcommand{\Rtimes}{\R^\times}
\newcommand{\tr}{\mathop{\mathrm{tr}}\nolimits}
\newcommand{\tup}[1]{\textup{#1}}
\newcommand{\rd}{\mathrm{d}}
\newcommand{\Z}{\mathbb{Z}}

\theoremstyle{plain}
\newtheorem{theorem}{Theorem}[section]
\newtheorem{proposition}[theorem]{Proposition}
\newtheorem{lemma}[theorem]{Lemma}
\newtheorem{corollary}[theorem]{Corollary}

\theoremstyle{definition}
\newtheorem{definition}[theorem]{Definition}
\newtheorem{remark}[theorem]{Remark}

\numberwithin{equation}{section}

\begin{document}
\title{Hirschman--Widder densities}

\author{Alexander Belton}
\address[A.~Belton]{Department of Mathematics and Statistics,
Lancaster University, Lancaster, UK}
\email{\tt a.belton@lancaster.ac.uk}

\author{Dominique Guillot}
\address[D.~Guillot]{University of Delaware, Newark, DE, USA}
\email{\tt dguillot@udel.edu}

\author{Apoorva Khare}
\address[A.~Khare]{Department of Mathematics, Indian Institute of
Science; and Analysis and Probability Research Group; Bangalore, India}
\email{\tt khare@iisc.ac.in}

\author{Mihai Putinar}
\address[M.~Putinar]{University of California at Santa Barbara, CA,
USA and Newcastle University, Newcastle upon Tyne, UK} 
\email{\tt mputinar@math.ucsb.edu, mihai.putinar@ncl.ac.uk}

\date{8th April 2022}

\keywords{P\'olya frequency function, totally positive function,
hypoexponential distribution}

\subjclass[2010]{15B48 (primary); %
05E05, 15A15, 30C40, 44A10, 47B34 (secondary)}

\begin{abstract}
Hirschman and Widder introduced a class of P\'olya frequency functions
given by linear combinations of one-sided exponential functions. The
members of this class are probability densities, and the class is
closed under convolution but not under pointwise multiplication. We
show that, generically, a polynomial function of such a density is a
P\'olya frequency function only if the polynomial is a homothety, and
also identify a subclass for which each positive-integer power is a
P\'olya frequency function. We further demonstrate connections between
the Maclaurin coefficients, the moments of these densities, and the
recovery of the density from finitely many moments, via Schur
polynomials.
\end{abstract}

\maketitle

\section{Introduction and main results}

The class of P\'olya frequency functions is central to the theory of
total positivity. Its basic properties were announced by Schoenberg in
1947--48~\cite{Schoenberg47,Schoenberg48} and further details were
provided in subsequent work \cite{Schoenberg50,Schoenberg51}. These
functions have been actively studied ever since.

\begin{definition}[Schoenberg]\label{DPFF}
A function $\Lambda : \R \to [ 0, \infty )$ is a \emph{P\'olya
frequency function} if it is Lebesgue integrable and non-zero at two
or more points, and the Toeplitz kernel
\[
T_\Lambda : \R \times \R \to \R; \ ( x, y ) \mapsto \Lambda( x - y )
\]
is totally non-negative. This last statement means that, for any
integer $p \geq 1$ and real numbers
\[
x_1 < \cdots < x_p \qquad \text{and} \qquad y_1 < \cdots < y_p,
\]
the matrix $\bigl( \Lambda( x_j - y_k ) \bigr)_{j, k = 1}^p$ has
non-negative determinant.
\end{definition}

Schoenberg showed in~\cite{Schoenberg51} that the bilateral Laplace
transform
\[
\cB\{ \Lambda \}( s ) := \int_\R e^{-x s} \Lambda( x ) \intd x
\]
of a P\'olya frequency function $\Lambda$ converges in a open strip
containing the imaginary axis, and equals on this strip the reciprocal
of an entire function $\Psi$ in the Laguerre--P\'olya
class~\cite{Laguerre1,Polya0}, with $\Psi( 0 ) = 1$. Conversely, any
function $\Psi$ of this form agrees with the reciprocal of the
bilateral Laplace transform of some P\'olya frequency function on its
strip of convergence. Schoenberg also proved that a P\'olya frequency
function necessarily has unbounded support, and either vanishes
nowhere or vanishes on a semi-axis. Members of the latter class of
functions are said to be \emph{one sided}, and Schoenberg
\cite{Schoenberg51} also characterized this subclass via the bilateral
Laplace transform. This characterization also allows the non-smooth
members of this subclass to be identified.

\begin{theorem}[Schoenberg]\label{TschoenbergPF}
If the one-sided P\'olya frequency function $\Lambda$ vanishes
on~$( -\infty, 0 )$ then $1 / \cB\{ \Lambda \}$ is the restriction of
an entire function $\Psi$ in the first Laguerre--P\'olya class, so has
the form
\begin{equation}\label{ELP}
\Psi( s ) = C e^{\delta s} \prod_{j = 1}^\infty ( 1 + \alpha_j s ), %
\quad \text{where } C > 0, \ \delta \geq 0, \ \alpha_j \geq 0 %
\text{ and } 0 < \sum_{j = 1}^\infty \alpha_j < \infty.
\end{equation}

Conversely, if the entire function $\Psi$ has the form
\tup{(\ref{ELP})} then there exists a P\'olya frequency function
$\Lambda$ that vanishes on $( -\infty, 0 )$ such that
$\Psi( s ) = 1 / \cB\{ \Lambda \}( s )$ on an open strip containing
the origin.

Such a P\'olya frequency function $\Lambda$ is continuous and positive
on~$( \delta, \infty )$ and vanishes on $( -\infty, \delta )$.
Furthermore, the function $\Lambda$ is smooth if and only if
$\alpha_j$ is non-zero for infinitely many $j$, and is continuous
unless~$\alpha_j$ is non-zero for exactly one $j$.
\end{theorem}

In this note, we study the one-sided P\'olya frequency functions which
are continuous but non-smooth, that is, those with at least two but
only finitely many non-zero terms $\alpha_1$, \ldots, $\alpha_m$ in
(\ref{ELP}), and their powers. We may normalize so that $\Lambda$ is a
probability density function, whence $C = 1$, and we may also assume
$\delta = 0$, by replacing $\Lambda$ with
$x \mapsto \Lambda( x + \delta )$. Thus, we study the collection of
\emph{finitely determined} P\'olya frequency functions of the form
$\Lambda_\balpha$, such that
\[
\cB\{ \Lambda_\balpha \}( s ) = %
\prod_{j = 1}^m ( 1 + \alpha_j s )^{-1}, \qquad %
\text{where } \balpha := ( \alpha_1, \ldots, \alpha_m ) %
\text{ and } m \geq 2.
\]

Some historical comments are appropriate, and here we recount this
subject's early developments in chronological order. In 1947,
Schoenberg~\cite{Schoenberg47} announced the notion of a P\'olya
frequency function. In their 1949 work, Hirschman and
Widder~\cite{HW49} studied $\Lambda_\balpha$ for distinct positive
$\alpha_1$, \ldots, $\alpha_m$ and its degree of smoothness, via the
Laplace transform. This was followed by Schoenberg's first full paper
on P\'olya frequency functions~\cite{Schoenberg51} in 1951. In this
work, Schoenberg placed the analysis of Hirschman and Widder in a
wider context, with the last part of Theorem~\ref{TschoenbergPF}
showing that the collection of Hirschman--Widder functions is dense,
in a suitable sense, in the set of non-smooth one-sided P\'olya
frequency functions. Finally, Hirschman and Widder's 1955
monograph~\cite{HW} contains a detailed analysis of these functions
and their Laplace transforms, and provides ample evidence for the
relevance of such functions to operational calculus and approximation
theory.

For this reason, we adopt the following
terminology for this family of functions, now allowing for
non-distinct and negative $\alpha_j$. For brevity, we let
$\Rtimes := \R \setminus \{ 0 \}$ denote the set of non-zero real
numbers.

\begin{definition}\label{DHWdensity}
Given $\balpha = ( \alpha_1, \ldots, \alpha_m ) \in (\Rtimes)^m$, where
$m \geq 2$, the corresponding \emph{Hirschman--Widder density} is the
unique continuous function $\Lambda_\balpha : \R \to [ 0, \infty )$
with bilateral Laplace transform
\begin{equation}\label{ELaplaceHW}
\int_\R e^{-x s} \Lambda_\balpha( x ) \intd x = %
\prod_{j = 1}^m ( 1 + \alpha_j s )^{-1}
\end{equation}
on the open half-plane
$\{ s \in \C : %
\Re s > -\alpha_j^{-1} \text{ for } j = 1, \ldots, m \}$.
\end{definition}

The next section focuses on some basic properties of these functions.

\begin{enumerate}
\item Each such function $\Lambda_\balpha$ exists and is unique.

\item The function $\Lambda_\balpha$ is both a P\'olya frequency
function and a probability density function. It is one sided if
and only if all the entries of $\balpha$ have the same sign.

\item The function $\Lambda_\balpha$ has a multiplicative
representation via convolution, as well as an additive
one involving one-sided exponentials.
The class of P\'olya frequency functions and its subclass of
Hirschman--Widder densities are both semigroups for the convolution
product.
\end{enumerate}

However, this collection of densities is not closed under pointwise
multiplication. The principal contribution of the present work is to
identify when some simple algebraic operations preserve the class of
Hirschman--Widder densities and when they do not. More specifically,
we focus on polynomial functions and study the generic behaviour of
these operations and their departure from mapping the class of
densities to itself.

We consider only those transforms which preserve the P\'olya frequency
property of infinite order, that is, with the natural number $p$
arbitrarily large in Definition~\ref{DPFF}. It was known to Karlin in the
1960s \cite{KarlinTAMS} that there exist P\'olya frequency functions
whose $\alpha$th powers are totally non-negative to some fixed finite
order for sufficiently large $\alpha$; see also~\cite[p.~115]{KK}. We
note that requiring total non-negativity only up to some finite order no
longer guarantees that the reciprocal of the Laplace transform is entire.

Our main result is as follows.

\begin{theorem}\label{T1sidedexample}\mbox{}\par
\begin{enumerate}
\item Suppose $m \geq 3$. There exists a subset $\cN$ of
$( 0, \infty )^m$ with Lebesgue measure zero such that
\[
p \circ \Lambda_\balpha : %
x \mapsto p\bigl( \Lambda_\balpha( x ) \bigr)
\]
is not a P\'olya frequency function for any
$\balpha \in ( 0, \infty )^m \setminus \cN$ and any real polynomial
$p$ that is not a homothety, that is, $p( x ) \not\equiv c x$ for any
$c > 0$.

\item Suppose
$\balpha = ( \alpha_1, \ldots, \alpha_m ) \in ( 0, \infty )^m$, where
$m \geq 2$, is such that the reciprocals $a_1 := \alpha_1^{-1}$,
\ldots, $a_m := \alpha_m^{-1}$ form an arithmetic progression. Then
$c \Lambda_{\balpha}^n$ is a P\'olya frequency function for every
$c > 0$ and every integer $n \geq 1$. If, moreover,
$\alpha_1 = \alpha_m$ or $\alpha_1 / \alpha_2$ is irrational, then
$p \circ \Lambda_\balpha$ is not a P\'olya frequency function for any
other real polynomial $p$.
\end{enumerate}
\end{theorem}

We remark that the first assertion with $m = 3$ and $p( x ) = x^n$ was
shown in our recent work \cite{BGKP-TN}, and played a key role in
characterizing those post-composition transforms that preserve the
class of one-sided P\'olya frequency functions.
Theorem~\ref{T1sidedexample} shows that these $m = 3$ examples are
merely the first in a large, multi-parameter family of P\'olya
frequency functions with the same property.

In the case where $p( x ) = x^n$, note that the assumption $m \geq 3$
in the first statement is necessary. Indeed, the $m = 2$ case is
covered by the second assertion, since every pair of numbers is
trivially in arithmetic progression.

The exceptional null set $\cN \subset (0,\infty)^m$ appearing in
Theorem~\ref{T1sidedexample}(1) has the following structure. If a tuple
$\balpha$ lies in the complement of $\cN$ in $(0,\infty)^m$ then $p \circ
\Lambda_\balpha$ is not a P\'olya frequency function for any
non-homothetic polynomial $p$.
The tuples in $\cN$ are precisely those of the form $( \alpha_1, \ldots,
\alpha_m )$ such that the coordinates of the reciprocal tuple $(
\alpha_1^{-1}, \ldots, \alpha_m^{-1} )$ are either $\mathbb{Q}$-linearly
independent or are roots of a countable family of non-zero polynomials
given in (\ref{Ef1f2}).
The $\mathbb{Q}$-linearly dependent tuples lie on a countable union of
hyperplanes, and the zero loci of the polynomials in~(\ref{Ef1f2}) also
lie on null sets.
The reciprocals of these tuples form the null set $\cN$.

As discussed in Section~\ref{Sprobability}, the Hirschman--Widder
densities are connected in multiple ways to classical probability theory.
In the present work, we link them to another area: the theory of
symmetric functions, and specifically Schur polynomials. In fact, this
also has a possible connection to probability theory: we show below that
these densities can be reconstructed from finitely many moments, via
symmetric function identities.

\begin{definition}\label{Dschur}
Given a field $\F$ of size at least $m \geq 2$, and a tuple
$\blambda = ( \lambda_1, \ldots, \lambda_m )$ of non-negative integers
$\lambda_1 \leq \cdots \leq \lambda_m$, we define the
corresponding \emph{Schur polynomial} to be the polynomial extension
of the function
\[
s_\blambda( a_1, \ldots, a_m ) := %
\frac{\det ( a_j^{\lambda_k} )_{j, k = 1}^m}{V( a_1, \ldots, a_m )}
\]
for distinct $a_1$, \ldots, $a_m \in \F$, where
$V( a_1, \ldots, a_m ) := \det( a_j^{k - 1} ) = %
\prod_{1 \leq j < k \leq m} ( a_k - a_j )$
is the usual Vandermonde determinant. If consecutive exponents are
equal, the Schur polynomial is identically zero.
\end{definition}

This definition differs from the one more commonly found in the
literature, such as \cite{Macdonald}, in that the entries of
$\blambda$ are non-decreasing rather than non-increasing and the
determinant in the numerator has exponent $\lambda_k$ instead of
$\lambda_k + n - k$. To switch between these two conventions is
straightforward: if $\nu_j = \lambda_{m - j + 1} - m + j$ then
$\bnu = ( \nu_1, \ldots, \nu_m )$ is such that
$\nu_1 \geq \cdots \geq \nu_m$ if and only if
$\lambda_1 < \cdots < \lambda_m$, in which case
$\widetilde{s}_\bnu( \ba ) = s_\blambda( \ba )$ for any
$\ba \in \F^m$, where $\widetilde{s}_\bnu$ is defined as in
\cite[(3.1)]{Macdonald}.

Schur polynomials are a distinguished family of symmetric functions,
and form a basis of homogeneous symmetric polynomials. They arise
naturally as characters of finite-dimensional irreducible
representations of $GL_{m + 1}( \C )$ or
$\mathfrak{sl}_{m + 1}( \C )$, and specialize to other families of
symmetric functions.

We can now make clear the connections between Hirschman--Widder
densities and Schur polynomials: both the Maclaurin coefficients and
the moments of these densities are given by Schur polynomials, and
are, in a certain sense, mirror images of one another.

\begin{theorem}\label{Tsymm}
Given $\balpha = ( \alpha_1, \ldots, \alpha_m ) \in ( 0, \infty )^m$,
where $m \geq 2$, the Hirschman--Widder density $\Lambda_\balpha$
is represented by its Maclaurin series on $[ 0, \infty )$,
with $n$th coefficient
\[
\Lambda_\balpha^{(n)}( 0^+ ) = \begin{cases}
 0 & \text{if } 0 \leq n \leq m - 2, \\
 ( -1 )^{n - m + 1} \alpha_1^{-1} \cdots \alpha_m^{-1}
s_{( 0, 1, \ldots, m - 2, n )}%
( \alpha_1^{-1}, \ldots, \alpha_m^{-1} ) & \text{if } n \geq m - 1.
\end{cases}
\]
The density $\Lambda_\balpha$ has $p$th moment
\[
\mu_p := \int_\R x^p \Lambda_\balpha( x ) \intd x = p! \, %
s_{( 0, 1, \ldots, m - 2, m - 1 + p )}(\alpha_1, \ldots, \alpha_m ) %
\qquad \text{if } p \geq 0.
\]
The parameter $\balpha$ can be recovered, up to permutation of its
entries, from the first $m$ moments, $\mu_1$, \ldots, $\mu_m$, and
also from the first $m + 1$ non-trivial Maclaurin coefficients,
$\Lambda_\balpha^{(m - 1)}( 0^+ )$, \ldots,
$\Lambda_\balpha^{(2 m - 1)}(0^+)$.
\end{theorem}

It may seem incongruous to require $m$ moments but $m + 1$ Maclaurin
coefficients in order to recover the $m$ coordinates $\alpha_1$, \ldots,
$\alpha_m$.
However, this is explained by noting that
$\Lambda_\balpha^{(n)}(0^+)$ equals $( \alpha_1 \cdots \alpha_m )^{-1}$
times a polynomial in the reciprocal coordinates $\alpha_1^{-1}$, \ldots,
$\alpha_m^{-1}$, and this polynomial is $1$ if $n=m-1$. Thus, the
parameter $\balpha$ is recovered from the data
$\Lambda_\balpha^{(n)}(0^+) / \Lambda_\balpha^{(m-1)}(0^+)$ for
$n = m$, \ldots, $2 m - 1$.

To the best of our knowledge, the connections in Theorem~\ref{Tsymm}
between the moments, the Maclaurin coefficients, and the moment-recovery
problem for Hirschman--Widder densities have not previously been noted in
the literature. While the moments are computable from first principles
using probability theory, we provide another recipe: the
moment-generating function may be obtained by evaluating the
generating function of the complete homogeneous symmetric
polynomials. In a similar spirit, the moment-recovery problem is seen
to be intimately connected with the Jacobi--Trudi identity.

P\'olya frequency functions, and in particular the subclass of
Hirschman--Widder densities, form a foundational chapter within the wider
framework of totally positive kernels. This latter concept continues to
attract generations of mathematicians, with surprising new developments.
It is not the intention of the present article to touch on the many
ramifications and current discoveries in this subject, with one exception.
The first footnote in the note by Vershik and Kerov \cite{VK} contains the
following line (our translation):
\textit{``It is worth mentioning that, in works going back to the 30s,
Schoenberg, but also Krein and Gantmacher, and later Karlin, have
developed the theory of totally positive kernels and matrices. However, a
connection with the characters of the unitary group and Weyl's formula
was not remarked at that time.''}
That happened later, with a series of spectacular discoveries: Fourier
transforms of P\'olya frequency functions were rediscovered as
irreducible characters of representations of the infinite symmetric group
\cite{Thoma,Ok}, and independently as irreducible characters of unitary
representations of the infinite unitary group $U(\infty)$,
\cite{Voiculescu,Pickrell}. Moreover, the classification and explicit
expression of spherical functions associated to classical groups
\cite{GN,hc} also pointed to the same class of totally positive kernels.
The string of coincidental findings of new facets of the same object does
not stop here: for example, challenging computations of the
characteristic functions of non-central Wishart distributions in
multivariate statistics led to P\'olya frequency functions and Schur
polynomials \cite{James,Farrell,Takemura}. Nowadays, these advances are
part of group representation theory \cite{OV,VK,Faraut} or random matrix
theory \cite{KK,KKS,KR,FKK}. By reversing the arrow of discovery, we
reproduce without proof in Section~\ref{Sorbital} an orbital integral
which encodes analytic properties of Hirschman--Widder densities. Our
work remains at an independent, elementary level, and we will indicate
the simplifications this view provides.

\subsection{Acknowledgements}

We are grateful to the referees for their careful reading of this paper
and several helpful suggestions that have improved its contents and
clarified our presentation.

D.G.~was partially supported by a University of Delaware Research
Foundation grant, by a Simons Foundation collaboration grant for
mathematicians, and by a University of Delaware strategic initiative
grant.

A.K.~was partially supported by Ramanujan Fellowship grant
SB/S2/RJN-121/2017, MATRICS grant MTR/2017/000295, SwarnaJayanti
Fellowship grants SB/SJF/2019-20/14 and DST/SJF/MS/2019/3 from SERB
and DST (Govt.~of India), and by grant F.510/25/CAS-II/2018(SAP-I)
from UGC (Govt.~of India).

M.P.~was partially supported by a Simons Foundation collaboration
grant.

\section{Hirschman--Widder densities and symmetric functions}

In this section, we prove Theorem~\ref{Tsymm}. We begin by recalling
two approaches to constructing Hirschman--Widder densities, guided by
the original memoir~\cite{HW}.

\subsection{Two constructions of Hirschman--Widder densities}

Following Hirschman and Widder \cite{HW}, we first establish the
existence of their eponymous densities via the convolution product.

\begin{proposition}\label{P1}
Let $\balpha = ( \alpha_1, \ldots, \alpha_m ) \in (\Rtimes)^m$, where
$m \geq 2$.
\begin{enumerate}
\item The corresponding Hirschman--Widder density $\Lambda_\balpha$
exists and is unique.

\item The function $\Lambda_\balpha$ is both a P\'olya frequency
function and a probability density function.

\item The function $\Lambda_\balpha$ is one sided if and only
if the entries of $\balpha$ all have the same sign.
\end{enumerate}
\end{proposition}

\begin{proof}
If $\alpha > 0$ then
\[
\varphi_\alpha : \R \to [ 0, \infty ); \ x \mapsto \begin{cases}
 0 & \text{if } x < 0, \\
 \alpha^{-1} e^{-\alpha^{-1} x} & \text{if } x \geq 0
\end{cases}
\]
is a P\'olya frequency function \cite[(3)]{Schoenberg47}, and if
$\alpha < 0$ then $\varphi_\alpha : x \mapsto \varphi_{-\alpha}( -x )$
is also a P\'olya frequency function. The function
\begin{equation}\label{Econvolution}
\Lambda_\balpha := %
\varphi_{\alpha_1} \ast \cdots \ast \varphi_{\alpha_m}
\end{equation}
is continuous, being a convolution product, and its bilateral Laplace
transform is as required, since
$\cB\{ \varphi_\alpha \}( s ) = ( 1 + \alpha s )^{-1}$.

The class of P\'olya frequency functions is closed under convolution
\cite[Lemma~5]{Schoenberg51}, and so~$\Lambda_\balpha$ is a
P\'olya frequency function. It is a probability density because
$\cB\{ \Lambda_\balpha \}( 0 ) = 1$.

For uniqueness, we note that any continuous function with prescribed
bilateral Laplace transform $F$ such that $t \mapsto F( \I t )$ is
integrable can be recovered everywhere via the Fourier--Mellin
integral: here,
\begin{equation}\label{Emellin}
\Lambda_\balpha( x ) = \frac{1}{2 \pi \I} %
\int_{-\I \infty}^{+\I \infty} %
\frac{e^{x s}}{( 1 + \alpha_1 s ) \cdots ( 1 + \alpha_m s )} %
\intd s \qquad ( x \in \R ).
\end{equation}

If $\alpha_j$ and $\alpha_k$ have opposite signs then a short calculation
shows that $( \varphi_{\alpha_j} \ast \varphi_{\alpha_k} )( x ) > 0$ for
any $x \in \R$. Furthermore, if $f$ is continuous and positive on $\R$
then $\varphi_\alpha \ast f$ is also continuous and positive on $\R$,
for any $\alpha \in \Rtimes$. This gives one part of (3) and the
converse is immediate.
\end{proof}

As the Laplace transform does not behave well for the product given by
pointwise multiplication, it is useful to have a second construction
for the function $\Lambda_\balpha$.

Given one-sided exponentials
\[
\bJ_{x \geq 0} \, e^{-a_1 x}, \ldots, \bJ_{x \geq 0} \, e^{-a_k x }, %
\qquad \text{where } k \geq 2 \text{ and } 0 < a_1 < \cdots < a_k,
\]
there are, up to homothety, only finitely many choices of real
coefficients $c_1$, \ldots, $c_k$ such that the linear combination
$\bJ_{x \geq 0} \sum_{j = 1}^k c_j e^{-a_j x}$ is a P\'olya frequency
function. More generally, we have the following result, where each
coefficient $c_j$ may be a polynomial.

\begin{definition}
For a tuple of positive real numbers $\ba = ( a_1, \ldots, a_k )$ such
that $k \geq 1$ and $a_1 < \cdots < a_k$, and a tuple of real
polynomials $\bc = ( c_1, \ldots, c_k )$, let
\[
\Lambda_{\ba, \bc} : \R \to \R; \ x \mapsto \begin{cases}
 \sum_{j = 1}^k c_j( x ) e^{-a_j x} & \text{if } x \geq 0, \\
 0 & \text{if } x < 0.
\end{cases}
\]
\end{definition}

We let $\deg p$ denote the degree of the polynomial $p$, with
$\deg p := -\infty$ if $p = 0$.

\begin{proposition}\label{Phirschmanwidder}
Suppose $\ba = ( a_1, \ldots, a_k )$ is a tuple of positive real numbers
such that $k \geq 1$ and $a_1 < \cdots < a_k$.
\begin{enumerate}
\item Given any tuple of non-negative integers
$\bn = ( n_1, \ldots, n_k )$, not all zero, there exists a unique
tuple $\bc = \bc^{\ba, \bn}
 = (\bc^{\ba, \bn}_1, \ldots, \bc^{\ba, \bn}_k)$
of real polynomials such that
\begin{equation}\label{EHWv2}
\cB\{ \Lambda_{\ba, \bc} \}( s ) = %
\prod_{j = 1}^k ( 1 + \alpha_j s )^{-n_j}, %
\quad \text{where } \alpha_j := a_j^{-1} %
\text{ for } j = 1, \ldots, k.
\end{equation}
The tuple $\bc^{\ba, \bn}$ is such that
$\deg c^{\ba, \bn}_j \leq n_j - 1$ for all $j$.

In particular, if
$\balpha = ( \alpha_1, \ldots, \alpha_m ) \in ( 0, \infty )^m$,
$\ba = ( a_1, \ldots, a_k ) \in ( 0, \infty )^k$ and
$\bm = ( m_1, \ldots, m_k )$ are such that $a_1 < \cdots < a_k$ and
$a_j^{-1}$ appears exactly $m_j$ times in $\balpha$, with
$m_1 + \cdots + m_k = m \geq 2$, then
the function $\Lambda_{\ba, \bc}$ with $\bc = \bc^{\ba, \bm}$ is the
Hirschman--Widder density $\Lambda_\balpha$.

\item If the tuple of real polynomials $\bc$ is not of the form
$t \bc^{\ba, \bn}$ for some $t > 0$ and some non-zero tuple of
non-negative integers $\bn$, then $\Lambda_{\ba, \bc}$ is not a
P\'olya frequency function.
\end{enumerate}
\end{proposition}

\begin{proof}
For (1), partial-fraction decomposition gives a family of real
coefficients $c^{\ba, \bn}_{j,l}$ such that
\[
\prod_{j = 1}^k ( 1 + \alpha_j s )^{-n_j} = %
\sum_{j = 1}^k \sum_{l = 1}^{n_j} %
c^{\ba, \bn}_{j, l} ( 1 + \alpha_j s )^{-l}.
\]
Furthermore, if $a > 0$ and $\alpha := a^{-1}$ then
\[
\cB\{ \bJ_{x \geq 0} \, x^{l - 1} e^{-a x} \}( s ) = %
( l - 1 )! \, \alpha^l ( 1 + \alpha s )^{-l} \qquad ( l = 1, 2, \ldots ).
\]
It follows that setting
\begin{equation}\label{Ecan}
c^{\ba, \bn}_j( x ) := \sum_{l = 1}^{n_j} %
\frac{a_j^l c^{\ba, \bn}_{j, l}}{( l - 1 )!} x^{l - 1} %
\qquad ( j = 1, \ldots k )
\end{equation}
gives the existence of $\bc^{\ba, \bn}$ as required. Uniqueness
follows by Lerch's theorem: if
$\cB\{ \Lambda_{\ba, \bc} \} \equiv \cB\{ \Lambda_{\ba, \bc'} \}$ on
some half plane then $\Lambda_{\ba, \bc} = \Lambda_{\ba, \bc'}$, as
both restrict to continuous functions on $[ 0, \infty )$ and vanish
elsewhere. Hence it suffices to show that
\[
\sum_{j = 1}^k p_j( x ) e^{-a_j x} \equiv 0 \implies %
p_1( x ) = \cdots = p_k( x ) \equiv 0
\]
for any real polynomials $p_1$, \ldots, $p_k$, but this follows
because
\[
\sum_{j = 1}^k p_j( x ) e^{-a_j x} \equiv 0 \implies %
p_1( x ) + \sum_{j = 2}^k p_j( x ) e^{( a_1 -a_j ) x} \equiv 0 %
\implies %
\lim_{x \to \infty} p_1( x ) = 0,
\]
whence $p_1 = 0$, and so on. For the final claim, it suffices to show
that $\Lambda_{\ba, \bc}$ is continuous at the origin, that is,
\[
0 = \sum_{j = 1}^k c_j^{\ba, \bm}( 0 ) = %
\sum_{j : m_j > 0} a_j c^{\ba, \bm}_{j, 1},
\]
where the second equality comes from (\ref{Ecan})
and the second sum is taken over those $j$ for which~$m_j$ is positive.
This sum vanishes because
\[
\sum_{j : m_j > 0} a_j c^{\ba, \bm}_{j, 1} = %
\lim_{s \to \infty} \sum_{j = 1}^k \sum_{l = 1}^{m_j}%
c^{\ba, \bm}_{j, l} s ( 1 + \alpha_j s )^{-l} = %
\lim_{s \to \infty} s \prod_{j = 1}^k ( 1 + \alpha_j s )^{-m_j} = 0.
\]

For (2), suppose $\bc$ is not of the form $t \bc^{\ba, \bn}$ as
asserted. Then $\cB\{ \Lambda_{\ba,\bc} \}( s )$ is a rational
function of the form
$p( s ) / \prod_{j = 1}^k ( 1 + \alpha_j s )^{m_j}$, with the
numerator polynomial $p$ not a factor of the denominator. Hence the
reciprocal of $\cB\{ \Lambda_{\ba, \bc} \}$ is not the restriction of
a function belonging to the Laguerre--P\'olya class. By
Theorem~\ref{TschoenbergPF}, $\Lambda_{\ba,\bc}$ is not a P\'olya
frequency function.
\end{proof}

\subsection{Vandermonde matrices and closed-form expressions}

In the first part of Theorem~\ref{T1sidedexample}, the generic tuple
$\balpha$ will be taken to have distinct entries. We may assume
without loss of generality that these are strictly decreasing, whence
the entries of the corresponding reciprocal tuple $\ba$ are strictly
increasing. The next result provides a closed-form expression for the
unique tuple $\bc$ such that $\Lambda_\balpha = \Lambda_{\ba, \bc}$ is
a Hirschman--Widder density. Note first that
Proposition~\ref{Phirschmanwidder} implies that $\bc$ consists of
polynomials of degree zero, that is, constants.

\begin{proposition}\label{P2}
Let $\balpha \in ( 0, \infty )^m$ be such that $m \geq 2$ and
$\alpha_1 > \cdots > \alpha_m$, and let $\ba = ( a_1, \ldots, a_m )$,
where $a_j := \alpha_j^{-1}$ for $j = 1$, \ldots, $m$. The
Hirschman--Widder density $\Lambda_\balpha = \Lambda_{\ba, \bc}$,
where $\bc = ( c_1, \ldots, c_m ) \in (\Rtimes)^m$ is such that
\[
c_j = a_j \prod_{k \neq j} \frac{a_k}{a_k - a_j} \qquad %
( j = 1, \ldots, m ).
\]
In particular, the constants $c_1$, \ldots, $c_m$ alternate in sign
and sum to zero.
\end{proposition}

Here we provide a different proof to that of Hirschman and Widder
\cite[Section~X.2.2]{HW} and so demonstrate a connection between these
densities and the theory of symmetric functions. We employ an
algebraic lemma involving alternating polynomials.

In the following definition and lemma, we let $\F$ denote an arbitrary
field.

\begin{definition}\label{Dhat}
Given any $\ba := ( a_1, \ldots, a_m ) \in \F^m$, where $m \geq 2$,
let $\widehat{\ba}_j \in \F^{m - 1}$ be obtained by removing the $j$th
term from $\ba$, so that
\[
\widehat{\ba}_j := %
( a_1, \ldots, a_{j - 1}, a_{j + 1}, \ldots, a_m ) \in \F^{m - 1} %
\qquad ( j \in 1, \ldots, m ).
\]
Recall that $V( \ba ) := \prod_{1 \leq j < k \leq m} ( a_k - a_j )$ is
the usual Vandermonde determinant.
\end{definition}

\begin{lemma}\label{Lvdm}
Given any $\ba \in \F^m$, with $m \geq 2$, the following identity holds
in the polynomial ring $\F[ X ]$:
\[
V( \ba ) X^l = %
\sum_{j = 1}^m ( -1 )^{j + l - 1} a_j^l V( \widehat{\ba}_j ) %
\prod_{k \neq j} ( X + a_k ) \qquad ( l = 0, \ldots, m - 1 ).
\]
\end{lemma}

\begin{proof}[Proof 1]
Suppose first that not all the entries of $\ba$ are distinct, say
$a_p = a_q$ for $p \neq q$. Then $V( \ba )$ vanishes, as do the
summands on the right-hand side for $j$ not equal to~$p$ or~$q$,
whereas the remaining two summands cancel each other, since
$\widehat{\ba}_q$ can be obtained from $\widehat{\ba}_p$ by
$| p - q | - 1$ transpositions.

We now assume that $\ba$ has distinct entries. Since both sides are
polynomials in $X$ of degree at most $m - 1$, it suffices to show they
agree at $-a_p$ for $p = 1$, \ldots, $m$. However, when evaluated
at~$-a_p$, the right-hand side reduces to
\[
( -a_p )^l V( \widehat{\ba}_p ) \prod_{k < p} ( a_p - a_k ) %
\prod_{k > p} ( a_k - a_p ),
\]
which is precisely $V( \ba ) ( -a_p )^l$.
\end{proof}

\begin{proof}[Proof 2]
An alternative argument, suggested by one of the referees, goes as
follows. The right-hand side of the proposed identity equals $(-1)^l \det
B$, as seen by expanding along the first column, where
\begin{align*}
B & = \begin{pmatrix}
a_1^l & X + a_1 & a_1 ( X + a_1 ) & \cdots & a_1^{m-2} ( X + a_1 )\\
a_2^l & X + a_2 & a_2 ( X + a_2 ) & \cdots & a_2^{m-2}( X + a_2 )\\
a_3^l & X + a_3 & a_3 ( X + a_3 ) & \cdots & a_3^{m-2}( X + a_3 )\\
\vdots & \vdots & \vdots & \ddots & \vdots\\
a_{m-1}^l & X + a_{m-1} & a_{m-1} ( X + a_{m-1} ) & \cdots &
a_{m-1}^{m-2} ( X + a_{m-1} )
\end{pmatrix}\\
 & = \begin{pmatrix}
1 & a_1 & \cdots & a_1^{m-1}\\
1 & a_2 & \cdots & a_2^{m-1}\\
\vdots & \vdots & \ddots & \vdots\\
1 & a_m & \cdots & a_m^{m-1}
\end{pmatrix}
( \be_{l+1}, X \be_1 + \be_2, X \be_2 + \be_3, \ldots, X \be_{m-1} +
\be_m )
\end{align*}
and $\{ \be_j : 1 \leq j \leq m \}$ is the standard basis of $\F^m$
written as column vectors.
Expanding the determinant of the final matrix along the first column
yields $(-1)^{( l + 1 ) + 1}$ times the determinant of a $2 \times 2$
block diagonal matrix, whose $(1,1)$ block is lower triangular with all
$l$ diagonal entries equal to $X$ and whose $(2,2)$ block is upper
triangular with all $m-l-1$ diagonal entries equal to $1$.
Hence, the right-hand side of the identity equals
\[
(-1)^l \det B = (-1)^l V(\ba) (-1)^l X^l = V(\ba) X^l,
\]
as claimed.
\end{proof}

\begin{remark}\label{Resp}
Lemma~\ref{Lvdm} can be applied to compute the inverse of the matrix
\begin{equation}\label{Esymfn}
E( \ba ) := \begin{pmatrix}
 e_0( \widehat{\ba}_1 ) & e_0( \widehat{\ba}_2 ) & \cdots & %
e_0( \widehat{\ba}_m ) \\
 e_1( \widehat{\ba}_1 ) & e_1( \widehat{\ba}_2 ) & \cdots & %
e_1( \widehat{\ba}_m ) \\
 \vdots & \vdots & \ddots  & \vdots \\
 e_{m - 1}( \widehat{\ba}_1 ) &  e_{m - 1}( \widehat{\ba}_2 ) & %
\cdots & e_{m - 1}( \widehat{\ba}_m )
\end{pmatrix},
\end{equation}
where $\ba = ( a_1, \ldots, a_m )$ contains distinct elements of the
field $\F$ and $e_l$ is the elementary symmetric polynomial
\begin{equation}\label{Eelemsymm}
e_0 \equiv 1 \qquad \text{and} \qquad %
e_l( b_1, \ldots, b_n ) := %
\sum_{1 \leq j_1 < j_2 < \cdots < j_l \leq n} %
b_{j_1} b_{j_2} \cdots b_{j_l}.
\end{equation}
These are Schur polynomials in the sense of Definition~\ref{Dschur},
with $e_{m - l}$ equal to $s_\blambda$ if
$\blambda := ( 0, 1, \ldots, l - 1, l + 1, \ldots, m )$; see
\cite[(3.9)]{Macdonald}.

A direct computation gives that
\begin{equation}\label{Evandermonde-inv}
E( \ba )^{-1} = ( -1 )^{m - 1} V( \ba )^{-1} D_\pm %
\begin{pmatrix}
 a_1^{m - 1} & a_2^{m - 1} & \cdots & a_m^{m - 1}\\
 a_1^{m - 2} & a_2^{m - 2} & \cdots & a_m^{m - 2}\\
 \vdots & \vdots & \ddots & \vdots\\
 1 & 1 & \cdots & 1
 \end{pmatrix} D_V D_\pm,
\end{equation}
where
\[
D_\pm := \diag( 1, -1, 1, -1, \ldots) %
\quad \text{and} \quad %
D_V := \diag\bigl( V( \widehat{\ba}_1 ), \ldots, %
V( \widehat{\ba}_m ) \bigr).
\]
The matrix identity~(\ref{Evandermonde-inv}) provides a closed-form
expression for the inverse of a standard Vandermonde matrix over any
field. This formula can also be deduced from an alternative expression
for the inverse; see~\cite{MS}.

A further consequence of (\ref{Evandermonde-inv}) is the following
expression for the determinant of~$E( \ba )$:
\begin{equation}
\det E( \ba ) = (-1)^{m ( m - 1 ) / 2} V( \ba ).
\end{equation}
This formula was proved by a different method
in~\cite[pp.~41--42]{Macdonald}.
\end{remark}

With Lemma~\ref{Lvdm} at hand, we now obtain the aforementioned
closed-form expression for the Hirschman--Widder density.

\begin{proof}[Proof of Proposition~\ref{P2}]
Let $\balpha$ and $\ba$ be as in the statement of the Proposition, and
define $\bc$ by letting
\begin{equation}\label{Evdm}
c_j := a_j \prod_{k \neq j} \frac{a_k}{a_k - a_j} = %
\frac{( -1 )^{j - 1} V( \widehat{\ba}_j )}{V( \ba )} %
\prod_{k = 1}^m a_k \qquad ( j = 1, \ldots, m ).
\end{equation}
The bilateral Laplace transform
\[
\cB\{ \Lambda_{\ba, \bc} \}( s ) = %
\sum_{j = 1}^m %
\frac{( -1 )^{j - 1} V( \widehat{\ba}_j)}{V(\ba) (s + a_j)} %
\prod_{k = 1}^m a_k = %
\frac{a_1 \cdots a_m}{( s + a_1 ) \cdots ( s + a_m )} = %
\cB\{ \Lambda_\balpha \}( s ),
\]
by Lemma~\ref{Lvdm} with $l = 0$ and $X = s$. We therefore conclude
that $\Lambda_\balpha = \Lambda_{\ba, \bc}$, by
Proposition~\ref{Phirschmanwidder}(1). That $c_1$, \ldots, $c_m$ are
alternating follows because $c_j$ has the same sign as
$( -1 )^{j -1}$, and that they sum to zero was shown in the proof of
Proposition~\ref{Phirschmanwidder}(1).
\end{proof}

\begin{remark}
The explicit form of $\Lambda_{\ba, \bc}$ obtained above provides a
connection to the topic of \emph{cardinal $L$-splines}; see
\cite{Micchelli}, for example, for more on these.
In~\cite[Section~5]{ST}, the authors study the restriction of a
certain $L$-spline to an interval $[ 0, \eta ]$. With a slight change
of notation to match that used here, this is
\[
\widetilde{A}_m( -x; t ) = %
1 + \sum_{j = 1}^m \frac{e^{-a_j x} ( 1 - t )}{t - e^{a_j \eta}} %
\prod_{k \neq j} \frac{a_k}{a_k - a_j}.
\]
It follows immediately from Proposition~\ref{P2} that the
Hirschman--Widder density $\Lambda_{\ba, \bc}$ is the asymptote of
this spline:
\[
\Lambda_{\ba, \bc}( x ) = %
\lim_{t \to \pm \infty} \frac{\rd}{\rd x} \widetilde{A}_m( -x; t ).
\]
Curry and Schoenberg \cite{CuS} conducted a study of P\'olya frequency
functions as limits of splines. Their approach was recently complemented
by Okounkov from the perspective of group representation theory: if
$\Omega$ is the orbit of $\diag( a_1, \ldots, a_m )$ under conjugation by
the unitary group $U(m)$ and $P$ is the projection from $\Omega$ to $\R$
given by taking the upper-left entry of each matrix then the the
fundamental spline $M_{m-1}( x; a_1, a_2, \ldots, a_m )$ can be viewed as
the density of the projection by $P$ of the normalized $U(m)$-invariant
measure on $\Omega$. See Section~$8$ in \cite{OV} for details, or
\cite{Faraut}.
\end{remark}

We now give a determinantal representation for generic Hirschman--Widder
densities, which follows directly from (\ref{Evdm}).

\begin{proposition}\label{Pclosedform}
Let $\balpha \in ( 0, \infty )^m$ have distinct coordinates and let $a_j
:= \alpha_j^{-1}$ for $j = 1$, \ldots, $m$, where $m \geq 2$.
The Hirschman--Widder density
\[
\Lambda_\balpha( x ) = \frac{a_1 \cdots a_m}{V(\ba)} \det \begin{pmatrix}
e^{-a_1 x} & e^{-a_2 x} & e^{-a_3 x} & \cdots & e^{-a_m x} \\
1 & 1 & 1 & \cdots & 1 \\
a_1 & a_2 & a_3 & \cdots & a_m\\
\vdots & \vdots & \vdots & \ddots & \vdots\\
a_1^{m-2} & a_2^{m-2} & a_3^{m-2} & \cdots & a_m^{m-2}
\end{pmatrix} \qquad (x \geq 0).
\]
\end{proposition}

\subsection{Non-smoothness at the origin}

We now take a closer look at Schoenberg's result
\cite[Corollary~2]{Schoenberg51} that Hirschman--Widder densities are
the only non-smooth, continuous one-sided P\'olya frequency functions,
up to homothety. In the process, we obtain what is, to the best of our
knowledge, a novel connection between P\'olya frequency functions and
symmetric functions.

The theme in this subsection is the smoothness of the map
$\balpha \mapsto \Lambda_\balpha( x )$, where $x$ is fixed; for
simplicity, we consider only $\balpha$ with positive entries. The
simplest example is such that
\[
\Lambda_{( \alpha_1, \alpha_2 )}( x ) = \begin{cases}
 ( \alpha_1 - \alpha_2 )^{-1} %
( e^{-\alpha_1^{-1} x} - e^{-\alpha_2^{-1} x} ) \qquad & %
\text{if } \alpha_1 \neq \alpha_2 \text{ and } x > 0, \\
 \alpha_1^{-2} x e^{-\alpha_1^{-1} x} & %
\text{if } \alpha_1 = \alpha_2 \text{ and } x > 0, \\
 0 & \text{otherwise}.
\end{cases}
\]
It is readily verified that the map
$( \alpha_1, \alpha_2 ) \mapsto \Lambda_{( \alpha_1,\alpha_2 )}( x )$
is continuous on $( 0, \infty )^2$ for any $x \in \R$. Indeed, more is
true.
\begin{enumerate}
\item For any non-negative integer $l$ and any $x \in \Rtimes$, the
map
\[
( 0, \infty )^2 \to \R; \ ( \alpha_1, \alpha_2 ) \mapsto %
\Lambda_{( \alpha_1, \alpha_2 )}^{(l)}( x )
\]
is real analytic.
\item The left and right limits at $0$ of the first derivative of
$\Lambda_{( \alpha_1, \alpha_2)}$ are distinct:
\begin{align*}
\Lambda_{( \alpha_1, \alpha_2 )}'( 0^- ) & :=
\lim_{x \to 0^-} \Lambda_{( \alpha_1, \alpha_2 )}'( x ) = 0 \\
\text{and} \quad %
\Lambda_{( \alpha_1, \alpha_2 )}'( 0^+ ) & :=
\lim_{x \to 0^+} \Lambda_{( \alpha_1, \alpha_2 )}'( x ) = %
\alpha_1^{-1} \alpha_2^{-1}.
\end{align*}
In particular, the Hirschman--Widder density
$\Lambda_{( \alpha_1, \alpha_2 )}$ is not differentiable at the origin.
\end{enumerate}

The following result is a generalization of the claim (1).

\begin{proposition}\label{Pderivative}
For any non-negative integer~$l$ and any $x \in \Rtimes$, the map
$\balpha \mapsto \Lambda_\balpha^{(l)}( x )$ is real analytic on
$( 0, \infty )^m$, for any $m \geq 2$.
\end{proposition}

\begin{proof}[Proof 1]
Let $\balpha = ( \alpha_1, \ldots, \alpha_m ) \in ( 0, \infty )^m$,
where $m \geq 2$, and recall the identity (\ref{Emellin}):
\[
\Lambda_\balpha( x ) = %
\frac{1}{2 \pi} \int_{-\infty}^\infty e^{\I t x} %
\prod_{j = 1}^m ( 1 + \I \alpha_j t )^{-1}  \intd t %
\qquad ( x \in \R ).
\]
There is no obstruction to extending this integral to the case where
$\alpha_1$, \ldots, $\alpha_m$ lie in the open right half-plane
$\ROHP := \{ z \in \C : \Re z > 0 \}$. Indeed, for such $\alpha_1$,
\ldots, $\alpha_m$, we have that
\[
| e^{\I t x} \prod_{j = 1}^m ( 1 + \I \alpha_j t )^{-1} | \leq 
| t |^{-m} \prod_{j = 1}^m | \Re \alpha_j |^{-1},
\]
which is Lebesgue integrable with respect to $t$. The analyticity of
the integrand in the complex variables $\alpha_1$, \ldots, $\alpha_m$
is then inherited by the integral.
 
We fix $x \in \Rtimes$. Integration by parts yields the identity
\[
( -\I x )^n \Lambda_\balpha( x ) = \frac{1}{2 \pi} %
\int_{-\infty}^\infty e^{\I t x} G_n( t ) \intd t
\]
for any non-negative integer $n$, where
\[
G_n( t ) := \frac{\partial^n}{\partial t^n} %
\Bigl( \prod_{j = 1}^m ( 1 + \I \alpha_j t )^{-1} \Bigr) = %
O( | t |^{-n - m} ) \text{ as } | t | \to \infty,
\]
since $G_n( t )$ is a homogeneous polynomial in
$( 1 + \I \alpha_1 t )^{-1}$, \ldots, $( 1 + \I \alpha_m t )^{-1}$ of
degree $n + m$.

Now suppose $n \geq l - m + 2$, and note that the derivative
\[
F_{l, n, \balpha}( x ) := \frac{\partial^l}{\partial x^l}%
\Bigl( ( -\I x )^n \Lambda_\balpha( x ) \Bigr) = %
\frac{1}{2 \pi} %
\int_{-\infty}^\infty ( \I t )^l e^{\I t x} G_n( t ) \intd t.
\]
Differentiation under the integral sign is valid here as long as the
function $t \mapsto | t |^l | G_n( t ) |$ is integrable, and this
holds because $l - n - m \leq -2$. The integrand is an analytic
function in the variables $\alpha_1$, \ldots, $\alpha_m$ as long as
these all lie in $\ROHP$, and an inductive argument now gives the
result, as $\Lambda_\balpha^{(l)}( x )$ can be expressed as a linear
combination of~$F_{l, n, \balpha}( x )$ and derivatives of lower
order.
\end{proof}

\begin{proof}[Proof 2]
If one uses the orbital-integral machinery and the
Harish-Chandra--Itzykson--Zuber integral, as explained in
Section~\ref{Sorbital}, then the representation~\eqref{Eorbit}
immediately gives the stronger assertion that, for fixed $l \geq 0$ and
$x \in \R^\times$, the map
$\balpha \mapsto \Lambda^{(l)}_\balpha( x )$ is complex analytic
in the half-space
$\{ \balpha \in \C^m : \Re \alpha_j > 0 \text{ for } j = 1, \ldots, m
\}$.
\end{proof}

We now examine the right-hand derivatives at the origin of our
Hirschman--Widder densities; all of the left-hand derivatives here are
clearly zero. Our approach will involve another instance of Schur
polynomials.

We have already seen that the elementary symmetric polynomials of
Remark~\ref{Resp} are Schur polynomials. Another well studied family
of symmetric functions is that of complete homogeneous symmetric
polynomials: if $l$ is a non-negative integer and
$\ba = ( a_1, \ldots, a_m ) \in \F^m$ for some $m \geq 1$ then
\begin{equation}
h_0 \equiv 1 \quad \text{and} \quad %
h_l( \ba ) := %
\sum_{1 \leq j_1 \leq j_2 \leq \cdots \leq j_l \leq m} %
a_{j_1} a_{j_2} \cdots a_{j_l}.
\end{equation}
These are also Schur polynomials: by \cite[(3.9)]{Macdonald}, we have
that
\begin{equation}\label{Eschurhom}
h_l = s_\blambda, \qquad \text{where } \qquad %
\blambda = ( 0, 1, \ldots, m - 2, m - 1 + l ).
\end{equation}

Having defined these polynomials, we may now state and prove the
following result.

\begin{proposition}\label{PhwMaclaurin}
Suppose that $\ba = ( a_1, \ldots, a_m ) \in ( 0, \infty )^m$, where
$m \geq 2$, and let
$\balpha := ( \alpha_1 = a_1^{-1}, \ldots, \alpha_m = a_m^{-1} )$. The
Hirschman--Widder density $\Lambda_\balpha$ has the Maclaurin-series
expansion
\begin{equation}\label{Emaclaurinhw}
\Lambda_\balpha( x ) = a_1 \cdots a_m %
\sum_{n = m - 1}^\infty %
\frac{( -1 )^{n - m + 1} h_{n - m + 1}( a_1, \ldots, a_m )}{n!} x^n
\end{equation}
valid for all $x \in [ 0, \infty )$. Consequently, the function
$\Lambda_\balpha$ is $m - 2$ times continuously differentiable, but
$\Lambda_\balpha^{( m - 1 ) }( 0 )$ does not exist.
\end{proposition}

\begin{proof}
We assume first that $a_1$, \ldots, $a_m$ are distinct and, without
loss of generality, that $a_1 < \cdots < a_m$. If $n$ is a
non-negative integer then Proposition~\ref{Pclosedform} gives that
\[
\Lambda^{( n )}_{\balpha}( 0^+ ) =
( -1 )^n \frac{a_1 \cdots a_m}{V( \ba )}
\det\begin{pmatrix}
 a_1^n & a_2^n & \cdots & a_m^n \\
 1 & 1 & \cdots & 1 \\
 a_1 & a_2 & \cdots & a_m \\
 a_1^2 & a_2^2 & \cdots & a_m^2 \\
 \vdots & \vdots & \ddots & \vdots \\
 a_1^{m - 2} & a_2^{m - 2} & \cdots & a_m^{m - 2}
\end{pmatrix}.
\]
If $n = 0$, $1$, \ldots, $m - 2$, then the matrix above has two
identical rows, so its determinant vanishes. For $n \geq m - 1$,
moving the top row to the bottom takes $m - 1$ transpositions, and now
the result follows from Definition~\ref{Dschur} and~(\ref{Eschurhom}),
together with the fact that $\Lambda_\balpha$ is the restriction to
$[ 0, \infty )$ of an entire function.

Since $\balpha \mapsto \Lambda_\balpha( x )$ is real analytic on
$( 0, \infty )^m$ for any $x > 0$, by Proposition~\ref{Pderivative},
it is continuous there. As the right-hand side of (\ref{Emaclaurinhw})
is also a continuous function of such~$\balpha$ for fixed $x > 0$, the
identity holds for all $( x, \balpha ) \in ( 0, \infty )^{m + 1}$. It
also holds trivially when $x = 0$.

An obstruction to $\Lambda_\balpha$ being continuously differentiable
can only appear at the origin, where $\Lambda_\balpha^{(n)}( 0^- )$ is
always zero. The working above shows that
$\Lambda_\balpha^{(n)}( 0^+ )$ is also zero if $n = 0$, \ldots,
$m - 2$, whereas
$\Lambda_\balpha^{( m - 1 )}( 0^+ ) = h_0( a_1, \ldots, a_m ) = 1$.
\end{proof}

As a brief digression, we use Proposition~\ref{PhwMaclaurin} to obtain a
classical identity in algebraic combinatorics,
Corollary~\ref{Cgeneratingfn}.
Although the identity is well known (see \cite[Chapter~I.2]{Macdonald} or
\cite{CS-paper}, for example), its connection to P\'olya frequency
functions is not.

We begin with the following elementary observation.

\begin{lemma}\label{LConvergenceRadius}
Let $\ba = ( a_1, \ldots, a_m ) \in [ 0, \infty )^m$, where
$m \geq 2$. Then
\[
\lim_{l \to \infty} h_l( \ba )^{1 / l} = %
\max\{ a_1, \ldots, a_m \}.
\]
\end{lemma}

\begin{proof}
If $a_k \geq a_j$ for $j = 1$, \ldots, $m$ and $l \geq m - 1$ then
\[
a_k^l \leq h_l( a_1, \ldots, a_m ) \leq %
\binom{l + m - 1}{l} a_k^l \leq %
\frac{( 2 l )^{m - 1}  a_k^l}{( m - 1 )!}.
\]
The result follows.
\end{proof}

\begin{corollary}\label{Cgeneratingfn}
Let $\ba = ( a_1, \ldots, a_m ) \in ( 0, \infty )^m$, where
$m \geq 2$. The generating function of the family of complete
homogeneous symmetric polynomials in $a_1$, \ldots, $a_m$ is
\[
\sum_{l = 0}^\infty h_l( a_1, \ldots, a_m ) z^l = %
\prod_{j = 1}^m \frac{1}{1 - a_j z} \qquad \text{whenever } %
| z | < \min\{ a_j^{-1} : j = 1, \ldots, m \}.
\]
\end{corollary}

As we explain below, this corollary provides a way to obtain the
moments of $\Lambda_\balpha$ from its Maclaurin coefficients. The
reverse inference can be drawn if one knows the moments, via arguments
of P\'olya \cite{Polya1915} and Schoenberg \cite{Schoenberg51}. This
is discussed further in Remark~\ref{Rmoments} below.

\begin{proof}
This power series has radius of convergence
$\min\{ a_j^{-1} : j = 1, \ldots, m \}$, by
Lemma~\ref{LConvergenceRadius}. Now suppose
$0 > z > \max\{ -a_j^{-1} : j = 1, \ldots, m \}$ and let
$s := -z^{-1}$. Setting $\alpha_j := a_j^{-1}$ for $j = 1$, \ldots,
$m$, we see that
\begin{align*}
\prod_{j = 1}^m \frac{1}{1 - a_j z} & = %
\alpha_1 \cdots \alpha_m s^m %
\cB\{ \Lambda_\balpha \}( s ) \\
 & = \alpha_1 \cdots \alpha_m  s^m %
\int_0^\infty e^{-s x} \Lambda_\balpha( x ) \intd x \\
& = s^m \int_0^\infty \sum_{n = m - 1}^\infty %
\frac{( -1 )^{n - m + 1} h_{n - m + 1}( a_1, \ldots, a_m )}{n!} %
x^n e^{-s x} \intd x,
\end{align*}
by Proposition~\ref{PhwMaclaurin}. As the series
\[
s^m \sum_{n = m - 1}^\infty %
\frac{h_{n - m + 1}( a_1, \ldots, a_m )}{n!} %
\int_0^\infty x^n e^{-s x} \intd x = %
\sum_{l = 0}^\infty h_l( a_1, \ldots, a_m ) s^{-l}
\]
is absolutely convergent, we may exchange the order of integration and
summation in the previous formula to see that the product
$\prod_{j = 1}^m ( 1 - a_j z )^{-1}$ equals
\[
\sum_{l = 0}^\infty h_l( a_1, \ldots, a_m ) z^l,
\]
as claimed. We conclude by the identity theorem.
\end{proof}

\subsection{Hirschman--Widder densities and
probability theory}\label{Sprobability}

We now explore some connections to probability theory. A natural first
question is to identify the random variables distributed with
Hirschman--Widder density functions.

\begin{proposition}\label{Pexpcomb}
Let $\alpha \in ( 0, \infty )^m$, where $m \geq 2$. Then the
Hirschman--Widder density~$\Lambda_\balpha$ is the probability density
function for the random variable
\[
\alpha_1 X_1 + \cdots + \alpha_m X_m,
\]
where $X_1$, \ldots, $X_m$ are independent and identically distributed
exponential random variables with mean $1$.
\end{proposition}

\begin{proof}
A exponential random variable $X$ with mean $1$ has density function
$\bJ_{x \geq 0} \, e^{-x}$, so if $\alpha > 0$ then $\alpha X$ has
density function
$\bJ_{x \geq 0} \, \alpha^{-1} e^{-\alpha^{-1} x} = %
\varphi_\alpha( x )$. The result now follows by
(\ref{Econvolution}).
\end{proof}

\begin{remark}
Hirschman--Widder density functions are studied in the probability and
statistics literature under the name of
\emph{hypoexponential densities}. They are intimately connected to
the time to absorption for a finite-state Markov chain. When the
entries of $\balpha$ are equal, then $\Lambda_\balpha$ is an
\emph{Erlang density}, named after the father of queueing theory;
this is a special case of the gamma distribution, occurring when the
shape parameter is an integer. These densities have found use in
diverse applied fields, including queueing theory, population
genetics, reliability analysis and cell biology. The connection to
Hirschman and Widder's memoir seems to be generally unnoticed in the
probability literature.
\end{remark}

The probabilistic interpretation also leads to closed-form
expressions. The explicit formula for $\Lambda_\balpha( x )$ when
$\alpha_1$, \ldots, $\alpha_m$ are positive and distinct, which
appears in Hirschman and Widder's memoir in analysis \cite{HW}, also
appears in probability textbooks; see, for example, Exercise~12 in
Chapter~I of Feller's 1966 book~\cite{Feller}. An explicit formula in
the case where repeats may occur was obtained by Amari and
Misra~\cite{AM}.

\begin{remark}
A second probabilistic interpretation of the density $\Lambda_\balpha$
can be derived from random matrix theory. Consider the diagonal matrix $D
= \diag( a_1, \ldots, a_m )$ and its orbit $\Omega$ under unitary
conjugation in the space of $m \times m$ positive semi-definite matrices.
If $\mu$ is the normalized $U(m)$-invariant measure on $\Omega$ then
$\Lambda_\balpha( x ) \intd x$ is the distribution of any diagonal entry
of a random positive semi-definite matrix of arbitrary size distributed
according to $\mu$.
See Section~3 and \cite[Section~8]{OV}, or \cite{Faraut} for details.
\end{remark}

\begin{remark}
Schoenberg's characterization of P\'olya frequency functions, as those
for which the reciprocal of the bilateral Laplace transform is the
restriction of an entire function in the Laguerre--P\'olya class,
admits the same probabilistic interpretation as in Proposition
\ref{Pexpcomb}. The Hadamard--Weierstrass factorization implies that
such a function is the density function of a possibly infinite linear
combination of independent and identically distributed exponential random
variables, together with at most one Gaussian random variable.
\end{remark}

We now obtain the promised closed form for the moments of the random
variables distributed according to Hirschman--Widder densities.

\begin{proof}[Proof of Theorem~\ref{Tsymm}]
The first part of this result has been established in the proof of
Proposition~\ref{PhwMaclaurin}. For the second, suppose first that the
entries of $\balpha = ( \alpha_1, \ldots, \alpha_m )$ are distinct,
and let $a_j := \alpha_j^{-1}$ for $j = 1$, \ldots, $m$. Suppose
further without loss of generality that $a_1 < \cdots < a_m$. If $p$
is a non-negative integer then (\ref{Evdm}) gives that
\begin{align*}
\frac{1}{p!} \int_{-\infty}^\infty x^p \Lambda_\balpha( x ) \intd x %
& = \frac{a_1 \cdots a_m}{p! \, V( \ba )} %
\sum_{j = 1}^m ( -1 )^{j - 1} V( \widehat{\ba}_j ) %
\int_0^\infty x^p e^{-a_j x} \intd x \\
 & = \frac{a_1 \cdots a_m}{V( \ba )} \sum_{j = 1}^m ( -1 )^{j - 1} %
V( \widehat{\ba}_j ) a_j^{-p - 1} \\
 & = \frac{a_1 \cdots a_m}{V( \ba )} %
\det\begin{pmatrix}
 a_1^{-p - 1} & a_2^{-p - 1} & \cdots & a_m^{-p - 1} \\
 1 & 1 & \cdots & 1 \\
 a_1 & a_2 & \cdots & a_m \\
 \vdots & \vdots & \ddots & \vdots \\
 a_1^{m - 2} & a_2^{m - 2} & \cdots & a_m^{m - 2}
\end{pmatrix}.
\end{align*}
This quantity is unchanged if, for $j = 1$, \ldots, $m$, we multiply
the $j$th column by $\alpha_j^{m - 2}$, and multiply the whole
expression by $( a_1 \cdots a_m )^{m - 2}$, to obtain
\begin{multline*}
\frac{( a_1 \cdots a_m )^{m - 1}}{V( \ba )} \det\begin{pmatrix}%
\alpha_1^{m - 1 + p} & \alpha_2^{m - 1 + p} & \cdots & %
\alpha_m^{m - 1 + p} \\
\alpha_1^{m - 2} & \alpha_2^{m - 2} & \cdots & \alpha_m^{m - 2} \\
\alpha_1^{m - 3} & \alpha_2^{m - 3} & \cdots &
\alpha_m^{m - 3} \\
\vdots & \vdots & \ddots & \vdots \\
1 & 1 & \cdots & 1
\end{pmatrix} \\
= \frac{( a_1 \cdots a_m )^{m - 1}}{V( \ba )} %
(-1)^{m ( m - 1 ) / 2} V(\balpha) %
s_{( 0, 1, \ldots, m - 2, m - 1 + p )}( \alpha_1, \ldots, \alpha_m ).
\end{multline*}
As
\[
V( \ba ) = %
( a_1 \cdots a_m )^{m - 1} ( -1 )^{m ( m - 1 ) / 2} V( \balpha ),
\]
this gives the result when $\balpha$ has distinct entries. The general
case now follows by a continuity argument.

Finally, we explain how to recover the parameter $\balpha$ from
moments or Maclaurin coefficients. The $p$th moment of
$\Lambda_\balpha$ is
\begin{equation}\label{Emoments}
\mu_p := \int_\R x^p \Lambda_\balpha( x ) \intd x = %
p! \, h_p( \balpha ) \qquad ( p = 1, 2, \ldots ).
\end{equation}
The Jacobi--Trudi identity~\cite[(3.4)]{Macdonald} asserts, for any
increasing tuple $( \lambda_1, \ldots, \lambda_m )$, that
\[
s_{( \lambda_1, \ldots, \lambda_m )}( \balpha ) = %
\det\bigl( h_{\lambda_{m - j + 1} - m + k}( \balpha ) %
\bigr)_{j, k = 1}^m = \det\bigl( h_{\lambda_j - k + 1}( \balpha ) %
\bigr)_{j, k = 1}^m
\]
with the conventions of Definition~\ref{Dschur} and $h_l := 0$
whenever $l < 0$. In particular,
\[
e_l( \balpha ) = %
s_{( 0, 1, \ldots, m - l - 1, m - l + 1, \ldots, m)}( \balpha ) = %
\det\begin{pmatrix} A_l & 0 \\ C_l & D_l \end{pmatrix} = \det D_l %
\qquad (l = 1, \ldots, m ),
\]
where $A_l$ is a lower-triangular matrix with $h_0( \balpha ) = 1$ as
each entry of the leading diagonal and $D_l$ is the $l \times l$
Toeplitz matrix
\[
\begin{pmatrix}
 h_1( \balpha ) & 1 & 0 & \cdots & 0 & 0 \\
 h_2( \balpha ) & h_1( \balpha ) & 1 & \cdots & 0 & 0 \\
 h_3( \balpha ) & h_2( \balpha ) & h_1( \balpha ) & \cdots & 0 & 0 \\
 \vdots & \vdots & \vdots & \ddots & \vdots & \vdots \\
 h_{l - 1}( \balpha ) & h_{l - 2}( \balpha ) & h_{l - 3}( \balpha ) & %
\cdots & h_1( \balpha ) & 1 \\
 h_l( \balpha ) & h_{l - 1}( \balpha ) & h_{l - 2}( \balpha ) & %
\cdots & h_2( \balpha ) & h_1( \balpha )
\end{pmatrix}.
\]
It follows that
$e_l( \balpha ) = \det D_l = f_l( \mu_1, \ldots, \mu_l )$ for
some polynomial function $f_l$. Hence
\begin{equation}\label{Epolynom}
F( t ) := 1 + \sum_{l = 1}^m f_l( \mu_1, \ldots, \mu_l ) t^l = %
\sum_{l = 0}^m e_l( \balpha ) t^l = %
\prod_{j = 1}^m ( 1 + \alpha_j t )
\end{equation}
is determined by the moments $\mu_1$, \ldots, $\mu_m$ and the roots of
$F$ yield precisely the entries of $\balpha$.

Similarly, Proposition~\ref{PhwMaclaurin} gives that
\[
\Lambda_\balpha^{(n)}( 0^+ ) / \Lambda_\balpha^{(m - 1)}( 0^+ ) = %
( -1 )^{n - m + 1} %
h_{n - m + 1}( \alpha_1^{-1}, \ldots, \alpha_m^{-1} ) %
\qquad ( n = m, \ldots, 2 m - 1 )
\]
and the previous working shows that we may recover, up to permutation
of its entries, the tuple $( \alpha_1^{-1}, \ldots, \alpha_m^{-1} )$
from these ratios of Maclaurin coefficients.
\end{proof}

\begin{remark}
The computation of the moments of $\Lambda_\balpha$ in the previous
proof was obtained with the assistance of the theory of symmetric
functions. A more direct approach, using the probabilistic
interpretation of $\Lambda_\balpha$, is also available. If $X_1$,
\ldots, $X_m$ are independent exponential random variables, each with
mean one, and $\alpha_1$, \ldots, $\alpha_m$ are positive constants,
then the random variable $X := \sum_{j = 1}^m \alpha_j X_j$ has
density $\Lambda_\balpha$, by Proposition~\ref{Pexpcomb}, and
moment-generating function
\[
\sum_{p = 0}^\infty \frac{\mu_p}{p!} z^p = \E[ e^{z X} ] = %
\cB\{ \Lambda_\balpha \}( -z ) = %
\prod_{j = 1}^m \frac{1}{1 - \alpha_j z}.
\]
Corollary~\ref{Cgeneratingfn}, with $\ba$ replaced by $\balpha$, now
shows that the moments are as claimed. Alternatively, we may proceed via
an explicit computation:
\[
\E\bigl[ \Bigl( \sum_{j = 1}^m \alpha_j X_j \Bigr)^p \bigr] = %
p! \, h_p( \alpha_1, \ldots, \alpha_m )
\]
for any non-negative integer $p$, since
\[
\Bigl( \sum_{j = 1}^m \alpha_j X_j \Bigr)^p = %
\sum_{p_1 + \cdots + p_m = p} \binom{p}{p_1 \cdots p_m} %
\alpha_1^{p_1} \cdots \alpha_m^{p_m} X_1^{p_1} \cdots X_m^{p_m},
\]
where the sum is taken over non-negative integers, and
$\E[ X_1^{p_1} ] = p_1!$.
\end{remark}

We now provide a connection between Hirschman--Widder densities and
certain P\'olya frequency sequences. Karlin, Proschan, and Barlow
proved that a probability density is a P\'olya frequency function if
and only if the sequence of normalized moments is a P\'olya frequency
sequence, in that the corresponding Toeplitz matrix is totally
non-negative; see \cite[Theorem~3]{KPB}. This Toeplitz matrix is
formed from a bi-infinite extension of the normalized-moments
sequence, which here is
\[
\ldots, 0, 0, 0, 1, h_1( \balpha ), h_2( \balpha ), h_3( \balpha),
\ldots
\]
and the total non-negativity of the corresponding Toeplitz matrix is
again the numerical shadow of the Jacobi--Trudi identity.

The observations above lead to the solution of the following moment
problem.

\begin{corollary}
Suppose $\bmu = ( \mu_1, \ldots, \mu_m ) \in \R^m$, where $m \geq 2$,
let $\mu_0 := 1$ and let the $l \times l$ Toeplitz matrix
\[
D_l := \begin{pmatrix}
 \mu_1 & 1 & 0 & \cdots & 0 & 0 \\
 \mu_2 / 2! & \mu_1 & 1 & \cdots & 0 & 0 \\
 \mu_3 / 3! & \mu_2 / 2! & \mu_1 & \cdots & 0 & 0 \\
 \vdots & \vdots & \vdots & \ddots & \vdots & \vdots \\
 \mu_{l - 1} / ( l - 1 )! & \mu_{l - 2} / ( l - 2 )! & %
\mu_{l - 3} / ( l - 3 )! & \cdots & \mu_1 & 1 \\
 \mu_l / l! & \mu_{l - 1} / ( l - 1 )! & \mu_{l - 2} / ( l - 2 )! & %
 \cdots & \mu_2 / 2! & \mu_1
\end{pmatrix}
\]
for $l = 1$, \ldots, $m$. The following are equivalent.
\begin{enumerate}
\item There exists
$\balpha = ( \alpha_1, \ldots, \alpha_m ) \in ( 0, \infty )^m$ such
that $\bmu$ is the truncated moment sequence of the Hirschman--Widder
density $\Lambda_\balpha$.

\item The generating polynomial
\[
F( t ) := 1 + \sum_{l = 1}^m t^l \det D_l
\]
of the determinant sequence $( \det D_1, \ldots, \det D_m )$ has all
of its roots in $( -\infty, 0 )$, so is in the first Laguerre--P\'olya
class.

\item The sequence $( b_n )_{n = -\infty}^\infty$ of the form
\[
\ldots, 0, 0, 0, b_0 = 1, \det D_1, \ldots, \det D_m, 0, 0, 0, \ldots
\]
is a P\'olya frequency sequence, in that
$\det( b_{p_j - q_k} )_{j, k = 1}^l$ is non-negative for any choice of
integers $p_1 < \cdots < p_l$ and $q_1 < \cdots < q_l$, where
$l \geq 1$.
\end{enumerate}
\end{corollary}

\begin{proof}
That $(1) \implies (2)$ follows from the proof of Theorem~\ref{Tsymm}.
For the reverse implication, it follows from $(2)$ that there exists
$\balpha \in ( 0, \infty )^m$ with $e_l( \balpha ) = \det D_l$ for
$l = 1$, \ldots, $m$. Expanding the determinant along its bottom row
shows that~$\mu_l$ can be recovered from $\det D_l$ and $\mu_1$,
\ldots, $\mu_{l - 1}$, so $\Lambda_\balpha$ has moments as required.

The equivalence $(2) \iff (3)$ follows from immediately from
\cite[Theorem~6]{AESW}.
\end{proof}

\begin{remark}\label{Rmoments}
We conclude this section with a line of enquiry motivated by
historical considerations. P\'olya~showed in his 1915 paper
\cite{Polya1915} that a function $\Psi$ in the Laguerre--P\'olya class
with the expansion
\begin{equation}\label{EreciprocalLP}
\frac{1}{\Psi( s )} = \mu_0 - \frac{\mu_1}{1!} s + %
\frac{\mu_2}{2!} s^2 - \frac{\mu_3}{3!} s^3 + \cdots
\end{equation}
and such that $\Psi( 0 ) > 0$ has Hankel determinant
\begin{equation}\label{hankel-det}
\det\begin{pmatrix}
\mu_0 & \mu_1 & \cdots & \mu_n \\
\mu_1 & \mu_2 & \cdots & \mu_{n + 1} \\
\vdots & \vdots & \ddots & \vdots \\
\mu_n & \mu_{n + 1} & \cdots & \mu_{2 n}
\end{pmatrix} > 0 \qquad \text{for } n = 0, 1, 2, \ldots.
\end{equation}
In 1920, Hamburger went the other way \cite{Hamburger1920} and deduced
the existence of an underlying density function $\Lambda$ from the
positivity of these determinants. This is precisely what led
Schoenberg to study these maps and to develop the theory of P\'olya
frequency functions. As observed by Schoenberg in~\cite{Schoenberg51},
the coefficient $\mu_p$ in (\ref{EreciprocalLP}) is precisely the
$p$th moment of the P\'olya frequency function $\Lambda$. When
$\Lambda = \Lambda_\balpha$, each normalized moment is a complete
homogeneous symmetric polynomial in the entries of $\balpha$, and so
this Hankel determinant is also a symmetric polynomial. The
Jacobi--Trudi identity shows that the Toeplitz moment determinant in
the proof of Theorem~\ref{Tsymm} is \emph{monomial positive} (a
positive linear combination of monomials) and even
\emph{Schur positive} (a positive linear combination of Schur
polynomials). It is natural to ask if something similar applies to the
Hankel moment determinant, but it turns out that neither property
holds in general.  For example, if $m = 2$ then
\[
\mu_p = p! \sum_{l = 0}^p \alpha_1^l \alpha_2^{p - l} \qquad %
\text{for } p = 0, 1, 2, \ldots,
\]
so
\[
\det\begin{pmatrix} \mu_0 & \mu_1 \\ \mu_1 & \mu_2 \end{pmatrix} = %
\alpha_1^2 + \alpha_2^2,
\]
which is monomial positive but not Schur positive, and
\[
\det\begin{pmatrix}
 \mu_0 & \mu_1 & \mu_2\\
 \mu_1 & \mu_2 & \mu_3\\
 \mu_2 & \mu_3 & \mu_4
\end{pmatrix} = %
4 \alpha_1^6 + 12 \alpha_1^4 \alpha_2^2 - %
8 \alpha_1^3 \alpha_2^3 + 12 \alpha_1^2 \alpha_2^4 + 4 \alpha_2^6,
\]
which is not even monomial positive.

We have that $\Psi( s ) = \prod_{j = 1}^m ( 1 + \alpha_j s )$, so the
formula (\ref{Emoments}) for the moments and the expansion
(\ref{EreciprocalLP}) give another proof of
Corollary~\ref{Cgeneratingfn}:
\[
\prod_{j = 1}^m ( 1 - \alpha_j s )^{-1} = \frac{1}{\Psi(-s)} = %
\sum_{l = 0}^\infty \frac{\mu_l}{l!} s^l = %
\sum_{l = 0}^\infty h_l( \alpha_1, \ldots, \alpha_m ) s^l.
\]
This line of thinking, together with (\ref{Epolynom}), also provides
the identity
\[
\sum_{l = 0}^\infty \frac{ \mu_l}{l!} s^l = %
\prod_{j = 1}^m ( 1 - \alpha_j s )^{-1} = %
\frac{1}{F( -s )} = %
\frac{1}{1 + \sum_{k = 1}^m f_k( \mu_1, \ldots, \mu_k ) ( -s )^k}.
\]
Assuming, as is common in some applied areas, that moments are the
only quantities available via measurements, we elaborate an alternative
reconstruction scheme which parallels well known algorithms used in
inverse problems. More specifically, we note that
$( \mu_j )_{j = 0}^\infty$, as a Stieltjes moment sequence, is
characterized by the positivity of the Hankel determinants of the form
(\ref{hankel-det}) as well as those with a shifted index,
\[
\det( \mu_{j + k + 1} )_{j, k = 0}^n > 0 %
\qquad \text{for } n = 0, 1, 2, \ldots.
\]
For a proof of this, see, for instance \cite[\S67]{Perron}. Next, we
note that checking whether the polynomial
$1 + \sum_{k = 1}^m f_k( \mu_1, \ldots, \mu_k ) ( -s )^k$ has only
real and negative roots can be achieved by using its Sturm sequence.

A classical result attributed to Kronecker asserts that the formal
series
\[
\sum_{j = 0}^\infty \gamma_j s^j
\]
represents a rational function if and only there exist positive
integers $n$ and $r_0$ such that
\[
\det( \gamma_{r + j + k} )_{j, k = 0}^n = 0 %
\qquad \text{for all } r \geq r_0,
\]
and the minimum value of $r_0$ is the degree of the denominator of the
rational function. In principle, to verify this criterion involves an
infinite sequence of vanishing Hankel determinants. Here, the hidden
positivity in the moments of a Hirschman--Widder distribution allows a
drastic reduction of Kronecker's criterion, to a single vanishing
determinant.

To this aim, and in order to identify the denominator
$\Psi( s ) = F( s )$ in the moment generating series without splitting
it into factors, we appeal to the theory of cumulants. The cumulant
series for $\Lambda_\balpha$ is
\[
K( s ) = \sum_{k = 1}^\infty \nu_k s^k := %
\log \sum_{l = 0}^\infty \frac{ \mu_l}{l!} s^l = %
-\sum_{j = 1}^m \log( 1 - \alpha_j s )
\]
and this is convergent whenever
$| s | < \min\{ \alpha_1^{-1}, \ldots, \alpha_m^{-1} \}$. Its
derivative has the series representation
\[
K'( s ) = \sum_{k = 1}^\infty k \nu_k s^{k - 1} = %
\sum_{j = 1}^m \frac{\alpha_j}{1 - \alpha_j s} = %
\sum_{j = 1}^m ( \alpha_j + \alpha_j^2 s + \alpha_j^3 s^2 + \cdots ),
\]
so the $k$th cumulant
\begin{equation}\label{Ecumulant}
\nu_k = \frac{1}{k} \sum_{j = 1}^m \alpha_j^k %
\qquad ( k = 1, 2, 3, \ldots ).
\end{equation}
As $K'$ is the Stieltjes transform of finitely many point masses, the
cumulant-generating function admits the representation
\[
K( s ) = %
\int_0^\infty \log\Bigl(\frac{t}{t - s}\Bigr) \intd\sigma( t ) %
\qquad ( | s | < \min\{ \alpha_1^{-1}, \ldots, \alpha_m^{-1} \} ),
\]
where $\sigma$ is the sum of unit point masses at $\alpha_1^{-1}$,
$\alpha_2^{-1}$, \ldots, $\alpha_m^{-1}$. In other words, the
cumulant-generating function $K$ coincides, up to a constant and for
small values of its argument, with the logarithmic potential of
equally distributed masses at the reciprocals of the entries
of~$\balpha$.

As a final step in our reconstruction process, we apply the Pad\'e
approximation scheme to the series representing $K'( s )$. We know
from Kronecker's criterion that the minimal choice of $n$ such that
$\det\bigl( ( k + j + 1 ) \nu_{j + k + 1} \bigr)_{j, k = 0}^n = 0$ is
when $n = m$. Elementary matrix-algebra operations single out a unique
pair of polynomials, a monic polynomial $P( z )$ with degree $m$
and~$Q( z )$ with degree $m - 1$, such that
\[
P( z ) \sum_{k = 1}^{m + 1} \frac{k \nu_k}{z^k}  = %
Q( z ) + O\Bigl( \frac{1}{z^2} \Bigr).
\]
Since $K'( s )$ is the Stieltjes transform of a positive measure, we
infer the equality of formal series
\[
\frac{1}{z} K'\Bigl( \frac{1}{z} \Bigr) = \frac{Q( z )}{P( z )}.
\]
Details of the algorithmic aspects of this derivation are due to
Stieltjes. They are masterfully exposed in Chapter~9 of \cite{Perron}.
It follows that
\[
\frac{F'( -s )}{F( -s )} = %
K'( s ) = \frac{s^{-1} Q( s^{-1} )}{P( s^{-1} )},
\]
and so we obtain the identity
\[
1 + \sum_{k = 1}^m f_k( \mu_1, \ldots, \mu_k ) ( -s )^k = %
s^m P( s^{-1} ).
\]
We stress that this Pad\'e approximation procedure identifies the
denominator in the moment generating series without computing its
zeros.
\end{remark}

Above, we touched on the old and eternally dominant theme of inversion
of the Laplace transform. An effective method of inversion for
rational functions without relying on simple-fraction decompositions
appears in \cite{Longman}. For a general overview of Laplace
transform-inversion, we refer to the monograph \cite{Cohen}.

\section{Orbital integrals}\label{Sorbital}

The Fourier transforms of P\'olya frequency functions can be viewed as
characteristic functions of certain unitarily invariant measures defined
on the space of infinite Hermitian matrices. In this section, we sketch
the basic framework and some key formulas, following the ample and
self-contained article of Olshanski and Vershik \cite{OV}. An alternate
reference, with complete proofs and a lucid global perspective on the
same topics is Faraut's article \cite{Faraut}.

Let $U(n)$ denote the compact group of unitary transformations of $\C^n$,
let $\Omega$ denote the orbit under unitary conjugation
$S \mapsto U S U^\ast$ of an $n \times n$ Hermitian matrix $S \in H(n)$,
and let $\mu$ denote the normalized $U(n)$-invariant measure carried by
$\Omega$. The Fourier transform or characteristic function of $\mu$ is
\begin{equation}\label{Eorbital}
f_S : H(n) \to \C; \ B \mapsto \int_\Omega e^{\I \tr(B M)} \mu( \rd M ).
\end{equation}
This function is invariant under unitary conjugation, so that
\[
f_S( U B U^\ast ) = f_S( B ) \qquad \text{for all } U \in U(n)
\text{ and } B \in H(n).
\]
Hence $f_S( B )$ depends only on the eigenvalues of $B$ and is a
symmetric function of these eigenvalues. Furthermore, as a Fourier
transform, the function $f_S$ is positive definite.

We now consider the inductive limit of such measures and functions
defined on the space $H(\infty) = \bigcup_n H(n)$ of Hermitian matrices
of arbitrary size. The functions, normalized by the condition $f( 0 ) =
1$, form a convex set and the extremal points of this set are
multiplicative, in the sense that
\[
f\bigl( \diag(b_1, b_2, \ldots, b_m) \bigr) = F( b_1 ) F( b_2 ) \cdots F( b_m )
\]
for some function of a real variable $F$. This situation occurs precisely
when the corresponding unitarily invariant measure $\mu$ on the union
$H(\infty)$ is ergodic.
The main classification theorem of \cite{OV}, Theorem~2.9, asserts the
existence of a bijective correspondence between ergodic, unitarily
invariant probability measures on $H(\infty)$ and P\'olya frequency
functions. To be more precise, $F$ is the Fourier transform of a P\'olya
frequency function attached to the ergodic measure $\mu$. Moreover,
specific invariant measures provide the building blocks of the class of
P\'olya frequency functions \cite[Corollaries~2.5 to~2.7]{OV}.

Now suppose $S = \diag( a_1, \ldots, a_m )$, where $a_1$, \ldots, $a_m$
are positive, and let $B = E_{11} = \diag( 1, 0, 0, \ldots, 0 )$. Passing
lightly over the technicalities required to extend~$f_S$, and using the
symmetry $f_S( \I x B ) = f_B( \I x S )$, which follows from the tracial
property, we have that
\begin{equation}\label{Eorbit}
f_S( \I x E_{11} ) = \int_{\Omega'} \exp\Bigl(-x  \sum_{j=1}^m a_j
|z_j|^2 \Bigr) \sigma(\rd z) \qquad ( x > 0 )
\end{equation}
where $\Omega' = S^{2 m - 1} \cong U( m )/ U( m - 1 )$ is the unit sphere
in $\C^m$ and $\sigma$ is the normalized rotationally invariant (uniform)
measure on the sphere.

We claim that, up to proper normalization, the above spherical average is
equal, to the Hirschman--Widder distribution $\Lambda_\balpha$ at the
point~$x$, where $\balpha = ( a_1^{-1}, \ldots, a_m^{-1} )$.
The alert reader will recognize the expression (\ref{Eorbit}) as a
Harish-Chandra--Itzykson--Zuber integral \cite{hc,iz}. A variety of
similar integral representations have recently been proposed as central
technical ingredients in random matrix theory \cite{KK,KKS,KR,FKK}.

\subsection{Explicit formulas}

The aim of this brief subsection is to describe an explicit link between
the spherical integral $f_S(B)$ above and Hirschman--Widder densities.
Further details, including complete proofs, are contained in
\cite[Section 5]{OV}.
Let $\ba = ( a_1, \ldots, a_m ) \in \R^m$ and $\bb = ( b_1, \ldots, b_m )
\in \C^m$ have corresponding diagonal matrices $S = \diag \ba \in H(m)$
and $B = \diag \bb$ respectively.
As in (\ref{Eorbital}) we let $f_S$ denote the characteristic function of
the $U(m)$-invariant probability measure $\mu$ with support $\Omega$,
where $\Omega$ is the $U(m)$-orbit of $S$ under conjugation:
\[
f_S(B) := \int_{U(m)} e^{\I \tr(B U S U^*)}\ \intd U = \int_\Omega 
e^{\I \tr(B M)} \mu( \rd M ).
\]
Since $f_S$ is entire and symmetric as a function of the coordinates of
$\bb$, it admits a Taylor-series expansion that is convergent everywhere,
so also a convergent expansion in terms of Schur polynomials:
\[
f_S( \diag \bb ) = \sum_\nu c_\nu s_\nu( \bb ),
\]
where the sum runs over Young diagrams with at most $m$ rows. A
computation by Olshanski and Vershik using characters of $U(m)$ and
change-of-bases formulas between symmetric power-sum polynomials and Schur
polynomials provides closed-form expressions for the coefficients
$c_\nu$: see \cite[Theorem~5.1]{OV}. Indeed, this strategy
appeared in explicit computations of Gel'fand and Naimark \cite{GN}, and
quite remarkably in multivariate statistics: see James \cite{James} and
the comments in \cite{Farrell}. From here, one derives the following
expansions: see \cite[Corollaries~5.2 and~5.4]{OV}.

\begin{proposition}\label{POV}
If the tuples $\ba = ( a_1, \ldots, a_m ) \in \R^m$ and $\bb = ( b_1,
\ldots, b_m ) \in \C^m$ each have distinct coordinates and $S = \diag
\ba$ then the orbital integral $f_S$ is given by the
Harish-Chandra--Itzykson--Zuber formula:
\[
f_S( -\I \diag \bb ) = \frac{\prod_{j = 0}^{m - 1} j!}{V( \ba ) V( \bb )}
\det\begin{pmatrix}
e^{b_1 a_1} & e^{b_1 a_2} & \cdots & e^{b_1 a_m}\\
e^{b_2 a_1} & e^{b_2 a_2} & \cdots & e^{b_2 a_m}\\
 \vdots & \vdots & \ddots & \vdots \\
e^{b_m a_1} & e^{b_m a_2} & \cdots & e^{b_m a_m}
\end{pmatrix}.
\]
If instead $B = \diag( 1, 0, \ldots, 0 ) = E_{11}$ then
\[
f_S( -\I x B ) = ( m - 1 )! \sum_{j = 0}^\infty \frac{h_j( a_1, \ldots,
a_m )}{( j + m - 1 )!} x^j,
\]
where $h_j$ is the $j$th complete homogeneous symmetric polynomial.
\end{proposition}

In particular, if $a_1$, \ldots, $a_m$ are positive and distinct, and $x
> 0$ then, by the second part of Proposition~\ref{POV} and
(\ref{Emaclaurinhw}),
\[
f_S( \I x E_{11} ) = %
( m - 1 )! ( -x )^{1 - m} \sum_{n = m - 1}^\infty \frac{h_{n - m + 1}(
a_1, \ldots,  a_m )}{n!} ( -x )^n = %
\frac{( m - 1 )! x^{1 - m}}{a_1 \cdots a_m} \Lambda_\balpha( x ),
\]
where $\balpha = ( a_1^{-1}, \ldots,  a_m^{-1} )$.

In conclusion, the Hirschman--Widder density possesses the following
integral and determinantal representations: if $x > 0$ 
and $\sigma(\rd z)$ denotes the normalized uniform measure on the sphere
$S^{2m-1}$, then
\begin{align}
\Lambda_\balpha( x ) & = \frac{a_1 \cdots a_m}{( m - 1 )!} x^{m - 1} %
\int_{S^{2 m - 1}} \exp\Bigl(-x \sum_{j=1}^m a_j |z_j|^2 \Bigr) \sigma(
\rd z) \\
& = \frac{a_1 \cdots a_m}{V( \ba )} \det\begin{pmatrix}
e^{-a_1 x} & e^{-a_2 x} & e^{-a_3 x} & \cdots & e^{-a_m x} \\
1 & 1 & 1 & \cdots & 1 \\
a_1 & a_2 & a_3 & \cdots & a_m\\
\vdots & \vdots & \vdots & \ddots & \vdots\\
a_1^{m - 2} & a_2^{m - 2} & a_3^{m - 2} & \cdots & a_m^{m - 2}
\end{pmatrix}.
\end{align}
The second representation can be obtained from the first identity in
Proposition~\ref{POV} by taking $\bb = ( -x, 0, y_1, \ldots, y_{m - 2} )$
and taking successively the $j$th partial derivative at zero with respect
to $y_j$ for $j = 1$ to $m - 2$.
Thus one has an alternative, shorter route for proving many of the
results contained in the previous section, if one accepts from the
beginning these two complementary formulas for the Hirschman--Widder
density~$\Lambda_\balpha$.
We leave the choice between this and our self-contained approach to the
reader.

\section{Proof of the main result}

With the results of the preceding section at hand, we now prove the
main result of this paper. We first introduce some notation which will be
used in the proof and in a later result.

\begin{definition}\label{Dnotation}
Given an integer $m \geq 2$ and a finite set $K$ of positive integers,
we let
\[
\Delta_m( K ) := \bigcup_{k \in K} \Delta_m( k ),
\]
where
\[
\Delta_m( k ) := %
\{ \bj = ( j_1, \ldots, j_m ) \in \Z^m : %
j_1, \ldots, j_m \geq 0, \ j_1 + \cdots + j_m = k \}
\]
is the set of integer-lattice points in the scaled standard
$m$-simplex. For any $k \geq 1$, any
$\bj = ( j_1, \ldots, j_m )  \in \Delta_m( k )$ and any
$\ba = ( a_1, \ldots, a_m ) \in \R^m$, we let
\[
\binom{k}{\bj} = \frac{k!}{\prod_{l = 1}^m j_l!}, \qquad
\ba^\bj := \prod_{l = 1}^m a_l^{j_l}, \quad  \text{and} \quad
\bj \cdot \ba := \sum_{l = 1}^m j_l a_l.
\]
\end{definition}

\begin{proof}[Proof of Theorem~\ref{T1sidedexample}]

\textit{Part }(1).
We note first that if the polynomial $p$ is such that $p( 0 ) \neq 0$
and $\Lambda_\balpha$ is a Hirschman--Widder density for some
$\balpha \in ( 0, \infty )^m$, then $p \circ \Lambda_\balpha$ equals
$p( 0 )$ on $( -\infty, 0 )$, so is not integrable. Also, we have that
$c \Lambda_\balpha$ is negative on $( 0, \infty )$ if $c < 0$. Thus we
need only to consider polynomials of degree at least two with zero
constant term for the remainder of this proof.

Suppose first that $m \geq 4$. To construct the null set
$\cN \subset ( 0, \infty )^m$, we begin as follows. For any integer
$n \geq 2$, we let
\[
\cK_n := \{ K \subset \{ 1, \ldots, n \} : n \in K \}
\]
denote the set of subsets of $\{ 1, \ldots, n \}$ containing $n$ and,
for any $K \in \cK_n$, we define the non-zero multivariate polynomial
$P_K$ by setting $P_K = f_1 - f_2$, where
\begin{equation}\label{Ef1f2}
f_i( \ba ) := ( -1 )^{( i - 1 ) n} V( \widehat{\ba}_i )^n %
\prod_{\bj \in \Delta_m( K ) \setminus \{ n \be_i \}} %
( \bj - n \be_i ) \cdot \ba \qquad ( \ba \in \R^m, \ i = 1, 2 ),
\end{equation}
with $\widehat{\ba}_1 = ( a_2, a_3, \ldots, a_m )$ and
$\widehat{\ba}_2 = ( a_1, a_3, \ldots, a_m )$, in accordance with
Definition~\ref{Dhat}, and
$\{ \be_1, \be_2, \ldots, \be_m \}$ being the standard basis
of~$\R^m$.

To see that $P_K \neq 0$, we observe that the $a_1$-degree of $f_1$
exceeds that of $f_2$, which gives the claim. To see this, we note
first that every factor in the product in $f_1$ has a linear
$a_1$-term, as $n$ is the maximum element of $K$, so
\[
\deg_{a_1} f_1 = | \Delta_m( K ) | - 1.
\]
However, the Vandermonde determinant $V( \widehat{\ba}_2 )$
contributes $m - 2$ linear $a_1$-terms to $f_2$ and therefore
\begin{align*}
\deg_{a_1} f_2 & = n ( m - 2 ) + | \Delta_m( K ) | - %
| \{ \bj \in \Delta_m( K ) : j_1 = 0 \} | \\
 & = n ( m - 2 ) + | \Delta_m( K ) | - | \Delta_{m - 1}( K ) |.
\end{align*}
Using the fact that $m \geq 4$ and $n \geq 2$, it follows that
\begin{align*}\label{Edegree}
\deg_{a_1} f_1 - \deg_{a_1} f_2 & = %
 | \Delta_{m - 1}( K )| - n ( m - 2 ) - 1 \\
 & \geq | \Delta_{m - 1}( n ) | - n ( m - 2 ) - 1 \\
 & = \binom{n + m - 2}{m - 2} - n ( m - 2 ) - 1 \\
 & > \frac{( n + m - 2 ) \cdots ( n + 2 ) ( n + 1 )}{( m - 2 )!} - %
( n + 1 ) ( m - 2 ) \\
 & \geq ( n + 1 ) \Bigl( %
\frac{m ( m - 1 ) \cdots 4}{( m - 2 )!} - ( m - 2 ) \Bigr) \\
 & = ( n + 1 ) \Bigl( \frac{m ( m - 1 )}{6} - ( m - 2 ) \Bigr) \\
 & = \frac{( n + 1 ) ( m - 3 ) ( m - 4 )}{6} \geq 0.
\end{align*}

Let $\cZ_K$ denote the zero locus of $P_K$ in $\R^m$, which is a null
set because $P_K$ is non-zero. With $S_m$ denoting the group of
permutations of $\{ 1, \ldots, m \}$, we let $\sigma \in S_m$ act on
subsets of $\R^m$ by permuting coordinates, so that
\[
\sigma( A ) := %
\{ ( a_{\sigma( 1 )}, \ldots, a_{\sigma( m )} ) : %
\ba = ( a_1, \ldots, a_m ) \in A \} %
\qquad \text{for any } A \subset \R^m,
\]
and note that this action is measure preserving. Finally, we let
\begin{align}
H_\bq & := \{ \bx \in \R^m : \bq \cdot \bx = 0 \}, \\
\tcN & := ( 0, \infty )^m \cap \biggl( %
\bigcup_{\bq \in \Q^m \setminus \{ \bZ \}} H_\bq \cup %
\bigcup_{\sigma \in S_m} \bigcup_{n = 2}^\infty %
\bigcup_{K \in \cK_n} \sigma( \cZ_K ) \biggr) \\
\text{and} \quad \cN & := %
\{ ( a_1^{-1}, \ldots, a_m^{-1} ) : \ba \in \tcN \}.
\end{align}
As a countable union of null sets, the set $\tcN$ is null.
Furthermore, the set $\cN$ is also null. To see this, we note that the
self-inverse map
\[
f : ( 0, \infty )^m \to ( 0, \infty )^m; \ %
( x_1, \ldots, x_m ) \mapsto ( x_1^{-1}, \ldots, x_m^{-1} )
\]
is Lipschitz when restricted to $[ l^{-1}, l ]^m$ for any positive
integer $l$, so preserves null sets there, and
\[
\cN = \bigcup_{l = 1}^\infty f\bigr( \tcN \cap [ l^{-1}, l ]^m \bigr).
\]

Let $\balpha \in ( 0, \infty )^m \setminus \cN$. Then the reciprocals
of the entries of $\balpha$ are linearly independent over $\Q$, since
they are contained in no hyperplane of the form $H_\bq$, so they are
distinct. Thus, we may find $\ba \in ( 0, \infty )^m \setminus \tcN$
such that $a_1 < \cdots < a_m$, the sets
$\{ a_1^{-1}, \ldots, a_m^{-1} \}$ and
$\{ \alpha_1, \ldots, \alpha_m \}$ are equal and
$P_K( \ba ) \neq 0$ for any $n \geq 2$ and any $K \in \cK_n$.

Now let~$\bc$ be as in Proposition~\ref{P2}, so that
$\Lambda_{\ba, \bc} = \Lambda_\balpha$. If
\[
c := \frac{a_1 \cdots a_m}{V( \ba )}
\]
then
\begin{equation}\label{Edentries}
\bd = ( d_1, \ldots, d_m ) := c^{-1} \bc = %
\bigl( V(\widehat{\ba}_1 ), -V( \widehat{\ba}_2 ), \ldots, %
\pm V( \widehat{\ba}_m ) \bigr).
\end{equation}
Let $\Lambda := \Lambda_{\ba, \bd} = c^{-1} \Lambda_\balpha$. If $p$
is a polynomial of degree at least two such that $p( 0 ) = 0$ then
$p \circ \Lambda_\balpha = q \circ \Lambda$, where the polynomial $q$
is such that $q( x ) = p( c x )$, so $q$ also has degree at least two
and no constant term. Thus it suffices to show that $p \circ \Lambda$
is not a P\'olya frequency function, where
$p( x ) = \sum_{k \in K} r_k x^k$, with $K \in \cK_n$ for some
$n \geq 2$ and $r_k \neq 0$ for all $k \in K$.

For any non-negative integer $k$, an explicit computation reveals that
the bilateral Laplace transform
\begin{align}\label{Etemp}
\cB\{ \Lambda^k \}( s ) & = %
\sum_{\bj \in \Delta_m( k )} \binom{k}{\bj} \cB\{ \bd^\bj
e^{ -( \bj \cdot \ba ) x} \}( s ) = %
\sum_{\bj \in \Delta_m( k )} \binom{k}{\bj}
\frac{\bd^\bj}{s + \bj \cdot \ba} = \frac{p_k( s )}{q_k( s )}
\end{align}
for polynomials $p_k$ and $q_k$, where the notation is as in
Definition~\ref{Dnotation} and
\[
q_k( s ) := \prod_{\bj \in \Delta_m( k )} ( s + \bj \cdot \ba ).
\]
Thus
\[
\cB\{ p \circ \Lambda \}( s ) = %
\sum_{k \in K} r_k \cB\{ \Lambda^k \}( s ) = %
\sum_{k \in K} r_k \frac{p_k( s )}{q_k( s )} = %
\frac{P( s )}{Q( s )},
\]
where
\[
Q( s ) := \prod_{k \in K} q_k( s ), \qquad %
\widehat{q}_k( s ) := \prod_{l \in K \setminus \{ k \}} q_l( s ) %
\quad \text{and} \quad %
P( s ) := \sum_{k \in K} r_k p_k( s ) \widehat{q}_k( s ).
\]
For any $k \in K$, the polynomial $q_k$ has the set of roots
$\{ -\bk \cdot \ba : \bk \in \Delta_m( k ) \}$. As the entries
of~$\ba$ are linearly independent over $\Q$, the roots of $Q$ are
simple. Furthermore, if $\bk \in \Delta_m( k )$ then
$\widehat{q}_j( s ) = 0$ whenever $j \in K \setminus \{ k \}$, so
\begin{equation}\label{Etemp2}
P( -\bk \cdot \ba ) = r_k p_k( -\bk \cdot \ba) \widehat{q}_k( -\bk \cdot
\ba).
\end{equation}
It follows from the definitions that
\[
\widehat{q}_k( -\bk \cdot \ba) = \prod_{ \bj \in \Delta_m( K ) \setminus
\Delta_m( k )} (\bj - \bk) \cdot \ba,
\]
and to compute $p_k( -\bk \cdot \ba)$, we use the final equality
in~(\ref{Etemp}) and the definition of $q_k$ to see that
\[
p_k( s ) = \sum_{\bj \in \Delta_m( k )} \binom{k}{\bj} \bd^\bj 
\prod_{\bj' \in \Delta_m( k ) \setminus \{ \bj \}} ( s + \bj' \cdot \ba ).
\]
Taking $s = -\bk \cdot \ba$ here and combining this with~(\ref{Etemp2})
shows that
\begin{equation}\label{EPpoly}
P( -\bk \cdot \ba ) = %
r_k \binom{k}{\bk} \bd^\bk
\prod_{\bj \in \Delta_m(K) \setminus \{ \bk \}} %
( \bj - \bk ) \cdot \ba %
\neq 0,
\end{equation}
again using $\mathbb{Q}$-linear independence. Thus $P$ does not vanish at
any root of $Q$, and so Theorem~\ref{TschoenbergPF} implies that $p \circ
\Lambda$ is not a P\'olya frequency function as long as $P$ is not
constant. However,
\begin{align*}
& P( -n \be_1 \cdot \ba ) - P( -n \be_2 \cdot \ba ) \\
 & = r_n \biggl( V( \widehat{\ba}_1 )^n %
\prod_{\bj \in \Delta_m( K ) \setminus \{ n \be_1 \}} %
( \bj - n \be_1 ) \cdot \ba + %
( -1 )^{n + 1} V( \widehat{\ba}_2 )^n %
\prod_{\bj \in \Delta_m( K ) \setminus \{ n \be_2 \}} %
( \bj - n \be_2 ) \cdot \ba \biggr) \\
 & = r_n P_K( \ba ) \neq 0,
\end{align*}
and this completes the proof whenever $m \geq 4$.

We now consider the case $m = 3$. When $p$ is a monomial, this was
resolved in previous work~\cite[Lemma~11.2]{BGKP-TN} whenever the
reciprocals of the entries of $\balpha$ are linearly independent over
$\Q$, so for $\Lambda$ as above. It remains to verify that
$p \circ \Lambda$ is not a P\'olya frequency function when
$p( x ) = \sum_{k \in K} r_k x^k$, with $r_k \neq 0$ for all $k \in K$
and $K \in \cK_n$ containing at least two elements. In this case, the
polynomial $P_K$ is non-zero, since we have that
\[
\deg_{a_1} f_1 - \deg_{a_1} f_2 = | \Delta_2( K )| - n - 1 > %
| \Delta_2( n ) | - n - 1 = 0.
\]
Thus we may proceed as above, as long as we take only $K$ containing
at least two elements in the definition of $\tcN$.

\textit{Part }(2). We consider first the case where
$\alpha_1 = \cdots = \alpha_m = \alpha$. The corresponding
Hirschman--Widder density $\Lambda_\balpha$ has bilateral Laplace
transform
\[
\cB\{ \Lambda_\balpha \}( s ) = ( 1 + \alpha s )^{-m},
\]
and inverting this transform gives that
\[
\Lambda_\balpha( x ) = \bJ_{x \geq 0} \frac{\alpha^{-m}}{( m - 1 )!} %
x^{m - 1} e^{-\alpha^{-1} x}.
\]
It follows immediately that, for any natural number $n$, the function
$\Lambda_\balpha^n$ is a positive multiple of $\Lambda_\bbeta$, where
$\bbeta = ( \alpha n^{-1}, \alpha n^{-1}, \ldots, \alpha n^{-1} ) \in %
( 0, \infty )^{n ( m - 1 ) + 1}$ has all its entries equal, and so
$\Lambda_\balpha^n$ is a P\'olya frequency function.

Now suppose that the polynomial $p$ is not a positive multiple of a
monomial. If $p$ has a constant term then $p \circ \Lambda_\balpha$ is
not integrable, so we may assume that
$p( x ) = \sum_{k \in K} r_k x^k$, where $K$ is a finite set of
natural numbers with at least two elements and $r_k \neq 0$ for all
$k \in K$. Then
\[
\cB\{ p \circ \Lambda_\balpha \}( s ) = %
\sum_{k \in K} r_k c_k ( 1 + \alpha k^{-1} s )^{-k ( m - 1 ) - 1} = %
\frac{P( s )}{Q( s )},
\]
where $c_k > 0$ for all $k \in K$,
\[
Q( s ) := \prod_{k \in K} ( 1 + \alpha k^{-1} s )^{k ( m - 1 ) + 1} %
\quad \text{and} \quad %
P( s ) := \sum_{k \in K} r_k c_k %
\prod_{j \in K \setminus \{ k \}} %
( 1 + \alpha j^{-1} s )^{j ( m - 1 ) + 1}.
\]
The polynomial $P$ is non-constant, since each of the terms in the sum
are polynomials with distinct positive degrees. Furthermore, the roots
of $Q$ are of the form $-k \alpha^{-1}$ for~$k \in K$, and
\[
P( -k \alpha^{-1} ) = r_k c_k %
\prod_{j \in K \setminus \{ k \}} ( 1 - j^{-1} k )^{j ( m - 1 ) + 1} %
\neq 0.
\]
Hence $Q( s ) / P( s )$ is not the restriction of an entire function
and so $p \circ \Lambda_\balpha$ is not a P\'olya frequency function,
by Theorem~\ref{TschoenbergPF}.

It remains to consider the case where $a_j = a_1 + ( j - 1 ) \delta$
for $j = 1$, \ldots, $m$, where $\delta$ is positive and independent
of~$j$. We consider the case $m = 2$ first, so that
\begin{equation}\label{E2tom}
\Lambda_\balpha^n( x ) = %
\bJ_{x \geq 0} \Bigl( \frac{a_1 a_2}{\delta} \Bigr)^n %
\sum_{j = 0}^n \binom{n}{j} (-1)^j e^{-( n a_1 + j \delta ) x}
\end{equation}
for any natural number $n$. If
$\bb := ( n a_1, n a_1 + \delta, \ldots, n a_1 + n \delta )$ and
$\bd = ( d_0, \ldots, d_n )$ is such that $\Lambda_{\bb, \bd}$ is a
Hirschman--Widder density, then Proposition~\ref{P2} gives that
\[
d_j = b_j \prod_{k \neq j} \frac{b_k}{b_k - b_j} = %
\frac{1}{n! \delta^n} \prod_{k = 0}^n ( n a_1 + k \delta ) %
( -1 )^j \binom{n}{j}.
\]
Thus $\Lambda_\balpha^n$ is a positive multiple of the
Hirschman--Widder density $\Lambda_{\bb, \bd}$, so is itself a P\'olya
frequency function.
An alternative argument for the $m=2$ case, pointed out to us by one
of the referees, goes as follows. Since
\[
\Lambda_\balpha^n( s ) = %
\bJ_{x \geq 0} \Bigl( \frac{a_1 a_2}{\delta} \Bigr)^n e^{-n a_1 x} ( 1 -
e^{-\delta x} )^n,
\]
an explicit computation shows that
\begin{align*}
\Bigl( \frac{\delta}{a_1 a_2} \Bigr)^n \cB\{ \Lambda_\balpha^n \}( s ) & = %
\frac{1}{\delta} \int_0^\infty e^{-( s + n a_1) x} (1 - e^{-\delta x})^n
\delta \intd x\\
 & = \frac{1}{\delta} \int_0^1 y^{( s + n a_1) / \delta} ( 1 - y )^n
 y^{-1} \intd y\\
 & = \frac{1}{\delta} \beta\bigl( (s + n a_1) / \delta, n + 1 \bigr),
\end{align*}
where $\beta$ denotes the beta function. From this it follows that
\[
\frac{1}{\cB \{ \Lambda_\balpha^n \}(s)} = \delta \left(
\frac{\delta}{a_1 a_2} \right)^n \frac{\Gamma( (s + n a_1)/\delta +
(n+1))}{n! \, \Gamma((s + n a_1)/\delta)},
\]
which is clearly a polynomial in $s$. Hence $\Lambda_\balpha^n$ is a
P\'olya frequency function for all $n \geq 1$.

For $m \geq 3$, let $\bbeta := ( b_1^{-1}, b_2^{-1} )$, where
$b_1 := a_1 / ( m - 1 )$ and $b_2 := b_1 + \delta = a_m / ( m - 1 )$.
Then, by (\ref{E2tom}) and the previous working, the function
$\Lambda_\bbeta^{m -1}$ is a positive multiple of the
Hirschman--Widder density $\Lambda_\balpha$. Hence $\Lambda_\balpha^n$
is a positive multiple of the P\'olya frequency function
$\Lambda_\bbeta^{n ( m - 1 )}$, and so is a P\'olya frequency function
itself.

Finally, suppose that $\alpha_1 / \alpha_2$ is irrational and let
$p( x ) = \sum_{k \in K} r_k x^k$, where $K$ is a finite set of
natural numbers with at least two elements and $r_k \neq 0$ for all
$k \in K$; as noted above, we need only consider $p$ of this form.
With the previous notation, we have that
$\Lambda_\balpha^k = c^k \Lambda_\bbeta^{( m - 1 ) k}$ for some
$c > 0$, and therefore
\[
\cB\{ p \circ \Lambda_\balpha \}( s ) = \sum_{k \in K} r_k c_k %
\prod_{j = 0}^{( m - 1 ) k} ( s + k a_1 + j \delta )^{-1} = %
\frac{P( s )}{Q( s )},
\]
where $c_k \neq 0$ for any $k \in K$,
\[
Q( s ) := \prod_{k \in K} \prod_{j = 0}^{( m - 1 ) k} %
( s + k a_1 + j \delta )
\]
and
\[
P( s ) := \sum_{k \in K} r_k c_k %
\prod_{v \in K \setminus \{ k \}} \prod_{u = 0}^{( m - 1 ) v} %
( s + v a_1 + u \delta ).
\]
The roots of $Q$ are simple, since if $( j, k )$ and $( j', k' )$ are
distinct then
\[
k a_1 + j \delta = k' a_1 + j' \delta \iff %
\bigl( ( k - j ) - ( k' - j' ) \bigr) a_1 + j a_2 = ( j' - j ) a_2 %
\implies \frac{\alpha_1}{\alpha_2} = \frac{a_2}{a_1} \in \Q.
\]
It follows that if $k \in K$ and
$j \in \{ 0, 1, \ldots, ( m - 1 ) k \}$ then
\[
P( -k a_1 - j \delta ) = r_k c_k %
\prod_{v \in K \setminus \{ k \}} \prod_{u = 0}^{( m - 1 ) v} %
\bigl( ( v - k ) a_1 + ( u - j ) \delta \bigr) \neq 0.
\]
Furthermore, the polynomial $P$ is the sum of polynomials with
distinct positive degrees, so is non-constant. We conclude from
Theorem~\ref{TschoenbergPF} that $p \circ \Lambda_\balpha$ is not a
P\'olya frequency function.
\end{proof}

We conclude with a discussion of the structure of the class of
polynomials mapping a fixed Hirschman--Widder density into the class
of P\'olya frequency functions.

\begin{proposition}\label{Phwpf}
Suppose the entries of
$\ba = ( a_1, \ldots, a_m ) \in ( 0, \infty )^m$ are linearly
independent over $\Q$ and strictly increasing, and let $\bc$ be as in
Proposition~\ref{P2}. For the polynomial
$p( x ) = \sum_{k \in K} r_k x^k$, where $K$ is a finite set of
non-negative integers and~$r_k \neq 0$ for all $k \in K$, the
following are equivalent.
\begin{enumerate}
\item $p \circ \Lambda_{\ba, \bc}$ is a P\'olya frequency function.
\item The function
\[
\eval_{\widetilde{P}} : \Delta_m( K ) \to \R; \ %
\bk \mapsto r_{| \bk |} \binom{| \bk |}{\bk} \bc^\bk %
\prod_{\bj \in \Delta_m( K ) \setminus \{ \bk \}} %
( \bj - \bk ) \cdot \ba
\]
is constant, where $| \bk | := k_1 + \cdots + k_m$.
\end{enumerate}
\end{proposition}

\begin{proof}
Following the proof of Theorem~\ref{T1sidedexample}(1), we have
that
$\cB \{ p \circ \Lambda_{\ba,\bc} \}(s) = \widetilde{P}( s ) / Q( s )$,
where
\[
Q( s ) = \prod_{\bk \in \Delta_m( K )} ( s + \bk \cdot \ba ) %
\quad \text{and} \quad %
\widetilde{P}( s ) = \sum_{\bk \in \Delta_m( K )} r_{| \bk |} %
\binom{| \bk |}{\bk} \bc^\bk %
\prod_{\bj \in \Delta_m( K ) \setminus \{ \bk \}} %
( s + \bj \cdot \ba ).
\]
As for $P$ above, the polynomial $\widetilde{P}$ does not vanish at
any root of $Q$. Thus $p \circ \Lambda_{\ba, \bc}$ is a P\'olya
frequency function if and only if $\widetilde{P}$ is a constant. Since
$\deg \widetilde{P} < | \Delta_m( K ) |$, it suffices to check that
evaluating $\widetilde{P}$ at the distinct points
$\{ -\bk \cdot \ba : \bk \in \Delta_m( K ) \}$ always yields the same
answer, but this is precisely~(2) above.
\end{proof}

\begin{remark}
We may further characterize when a polynomial
$p( x ) = \sum_{k \in K} r_k x^k$ maps the Hirschman--Widder density
$\Lambda_{\ba, \bc}$ into the class of all such densities. This
happens if and only if Proposition~\ref{Phwpf}(2) holds and the
function $p \circ \Lambda_{\ba, \bc}$ has unit integral. As
\[
( p \circ \Lambda_{\ba, \bc} )( x ) = %
\sum_{\bk \in \Delta_m( K )} r_{| \bk |} %
\binom{| \bk |}{\bk} \bc^\bk e^{-(\bk \cdot \ba) x},
\]
we obtain the additional condition
\[
\sum_{\bk \in \Delta_m( K )} r_{| \bk |} \binom{| \bk |}{\bk} %
\frac{\bc^\bk}{\bk \cdot \ba} = 1.
\]
\end{remark}



\end{document}